\documentclass[]{article}
\usepackage{amsfonts}
\usepackage{amsmath}
\usepackage{amssymb}
\usepackage{amsthm}
\usepackage{enumerate}
\usepackage[latin1]{inputenc} 
\usepackage[dvipsnames]{xcolor}
\usepackage{mathtools}
\usepackage{pifont}
\usepackage{hyperref}
\usepackage{tikz}
\usepackage{tikz-cd}
\usepackage{soul}
\newcommand*\circled[1]{\tikz[baseline=(char.base)]  {\node[shape=circle,draw, inner sep=0.5pt] (char) {#1};}}

\usepackage{cancel}
\usepackage[normalem]{ulem}

\newcommand*\circledD[1]{\tikz[baseline=(char.base)]{\node[shape=circle,draw,dashed, inner sep=0.5pt] (char) {#1};}}

\usepackage{parcolumns}

\newtheorem{theorem}{Theorem}

\newtheorem{corollary}[theorem]{Corollary}
\newtheorem{lemma}[theorem]{Lemma}

\theoremstyle{definition}

\newtheorem{definition}[theorem]{Definition}
\newtheorem{conjecture}[theorem]{Conjecture}

\usepackage{subfig}

\newcommand\tikznode[3][]%
   {\tikz[remember picture,baseline=(#2.base)]
      \node[minimum size=0pt,inner sep=0pt,#1](#2){#3};%
   }
\tikzset{>=stealth}

\usepackage[paperwidth=210mm,paperheight=300mm,
textwidth=170mm,textheight=250mm,top=25mm,lmargin=20mm
]{geometry}

\title{A simplification of the C-realizability criterion for the  Nonnegative Inverse Eigenvalue Problem for integers}

\author{Alberto Borobia\thanks{Departamento de Matem\'{a}ticas, Universidad Nacional de Educaci\'on a Distancia (UNED), email: {\tt aborobia@mat.uned.es}}\quad and\quad Roberto Canogar\thanks{Departamento de Matem\'{a}ticas, Universidad Nacional de Educaci\'on a Distancia (UNED), email: {\tt rcanogar@mat.uned.es}}}
\date{\today}

\begin{document}

\maketitle

\begin{abstract}
    A multiset $\Lambda=\{\lambda_1,\ldots,\lambda_n\}$ of complex  numbers is said to be realizable whenever there exists a nonnegative matrix of order $n$ with spectrum $\Lambda$.  One of the broadest criterion that guarantees realizability is the $C-$realizability. It says that $\Lambda$, with real numbers, is  $C-$realizable if it can be obtained starting from $n$ basic multisets $\{0\},\ldots,\{0\}$ by  successively applying any finite number of times any of the following  rules:  (a) join two of the multisets; (b) increase by $\epsilon>0$ the Perron root of one of the multisets;  (c) increase  by $\epsilon>0$ the Perron root of one of the multisets and simultaneously  increase or decrease by $\epsilon$ any other value of the same multiset.

If in the above rules we restrict $\epsilon$ to be an integer number, then we will always obtain  multisets of integers. And for integers, this work proves that the collection of original rules (a)--(c) is equivalent to a simplified collection: (a) join two of the multisets; (b') increase by 1 the Perron root of one of the multisets ; and (c') increase by 1 the Perron root of one of the multisets and simultaneously decrease by 1 a non-positive value of the same multiset. 
This simplification  is useful if we want to decide if a given  multiset of integers is $C-$realizable or not.

\end{abstract}

\noindent{\bf Keywords:} Nonnegative matrices; Nonnegative inverse eigenvalue problem; realizable spectra; realizability criteria; $C-$realizability.

\noindent{\bf AMS Subject Classification:} 15A18, 15A29.

\section{Introduction}
A matrix is nonnegative if all its entries are nonnegative numbers. A multiset  $\Lambda=\{\lambda_1,\ldots,\lambda_n\}$ of complex numbers is said to be realizable if there exists a nonnegative matrix with   spectrum $\Lambda$. In 1949 Suleimanova posed the nonnegative inverse eigenvalue problem (NIEP), which is  the problem of determining which multisets of complex numbers are realizable.  A thorough survey for the NIEP was provided by Johnson et al.~\cite{jmpp}. 

The  version for reals of the NIEP, the RNIEP,  asks for realizable multisets of real numbers. The RNIEP has a broad literature, where special attention has been focused on obtaining criteria of realizability. Here we are interested in one of such criteria, namely the $C-$realizability (see Borobia, Moro and Soto~\cite{bms2004,bms2008}). First, we give three classical results (the first and second are well known, and the third is  due to Guo).
\begin{theorem}\label{join}
   Let $\Lambda$ and $\Gamma$ be two realizable multisets of $\mathbb{C}$, then $\Lambda\cup \Gamma$ is realizable.  
\end{theorem}

\begin{theorem}\label{1+}
   Let $\{\lambda_1, \ldots,\lambda_n\}$  be a realizable multiset of $\mathbb{C}$ whose   Perron root is $\lambda_1$. Then for any $\epsilon>0$ the multiset $\{\lambda_1+\epsilon,\lambda_2,\ldots,\lambda_n\}$ is realizable .  
\end{theorem}

\begin{theorem}\label{guo_th} \cite{guo1997}
Let $\{\lambda_1, \ldots,\lambda_n\}$  be a realizable multiset of $\mathbb{C}$ whose   Perron root is $\lambda_1$ and let $\lambda_2$ be real. Then for any   $\epsilon >0$, the multiset $\{\lambda_1+ \epsilon,\lambda_2\pm \epsilon, \lambda_3 ,\ldots,\lambda_n\}$ is realizable.
\end{theorem}

Based on these theorems, a new criterion for the RNIEP was introduced in \cite{bms2008}. Theorems~\ref{join},~\ref{1+}, and~\ref{guo_th} can be applied to realizable multisets $\Lambda$ and $\Gamma$  of $\mathbb{C}$. But note that whenever $\Lambda$ and $\Gamma$ are multisets of $\mathbb{R}$, then the result of the rules is a multiset of $\mathbb{R}$ as well. In the next definition we start with the collection of realizable real multisets $\{0\}, \ldots, \{0\}$, so we will always obtain real multisets.

\begin{definition} \label{C_realizability_2}
A multiset $\{\lambda_1,\ldots,\lambda_n\}$ of $\mathbb{R}$ is  {\bf $C-$realizable}  if it can be reached starting from the $n$ realizable sets $\{0\}, \ldots, \{0\}$ and successively applying any finite number of times any of the  Theorems~\ref{join},~\ref{1+}, and~\ref{guo_th}.
\end{definition}

And the main result in \cite{bms2008} reads as follows.

\begin{theorem} \label{th_original} 
    If  a multiset   of  $\mathbb{R}$ is $C-$realizable, then it  is realizable.
\end{theorem}

Mariju\'an, Pisonero and Soto~\cite{mps2} (see also \cite{mps1})  showed that the $C-$realizability was one of the broadest criteria for the  RNIEP.  Ellard and \v{S}migoc~\cite{es2015} proved the equivalence of the $C-$realizability criterion to other three criteria due to Soto~\cite{s2013}, Soules~\cite{s1983} (refined by Elsner, Nabben and Neumann~\cite{enn1998}), and their own Ellard-\v{S}migoc method. As a consequence of this equivalence  they concluded that  the $C-$realizability is also a criterion of nonnegative symmetric realizability.
The advantage of $C-$realizability with respect to other criteria is the simplicity of its approach. The main objective of this work is to go further, that is, simplify the basic rules (given by Theorems~\ref{join},~\ref{1+}, and~\ref{guo_th}) that guarantee $C-$realizability.

 Recently, Mariju\'an and Moro~\cite{mm2021}  obtained a combinatorial characterization of $C-$realizable multisets with zero sum  (see~\cite{mm2023} for the general case), together with explicit formulas for $C-$realizable multisets having at most four positive entries.  It should be noted that  the definition of $C-$realizability  that  they consider  follows from a incomplete version of Theorem~\ref{guo_th}, more precisely, they omit the realizability of $\{\lambda_1+ \epsilon,\lambda_2+ \epsilon, \lambda_3 ,\ldots,\lambda_n\}$. Intuitively this makes sense because it was thought that its inclusion would not augment the set of realizable lists.  For integers that intuition was correct,  although it is quite hard to prove as we will see. On the other hand,  for rationals and for reals  we have not been able to prove it. These difficulties seem to tells us that something deeper than expected is at work here.

In Section~\ref{Another-approach-to-C} we will dedicate special attention to the rule given by Theorem~\ref{guo_th}, indeed we will partition this rule into four parts. So the three original rules  (corresponding to Theorems~\ref{join} to~\ref{guo_th}) become six rules after the partition. Consequently, we will introduce the $C_6$ and the $C_3-$realizability for a multiset, depending on whether we consider the six rules  or if we consider  only three of them: the one corresponding to Theorem~\ref{join}, the one corresponding to Theorem~\ref{1+}, and one of the four corresponding to Theorem~\ref{guo_th}. If the   rules only involve integers then we will  speak of $C_{\mathbb{Z},6}$ and $C_{\mathbb{Z},3}-$realizability. These types of realizability can also be extended to rationals or to reals. In Section~\ref{types_moves} we introduce states of integer numbers as tuples of tuples of integer numbers: for example, $((5,-5),(2,-1,-1))$ is a state. We provide six transformations or moves that  can be applied to states. The six moves for states will correspond to the six rules for multisets. We finish this section providing a collection of laws that any sequence of moves should verify. 

It is important to note that when we work with multisets the order of its elements is irrelevant, but when we work with states the order is relevant.
In Section~\ref{C*6-realizability and C*3-realizability} we introduce the $C^*_{\mathbb{Z},6}$ and the $C^*_{\mathbb{Z},3}-$realizability for tuples of integers. This concepts  are analogous to $C_{\mathbb{Z},6}$ and $C_{\mathbb{Z},3}-$realizability, considering the 6 moves instead of the 6 rules, and starting from $((0),\ldots,(0))$ instead of $\{0\},\ldots,\{0\}$. And we will finish this section by establishing the relation between $C_{\mathbb{Z},k}$ and $C^*_{\mathbb{Z},k}-$realizability for $k=6$ and for $k=3$.  

 In Section~\ref{Technical_Lemmas} we study the swap of two consecutive moves: let $M_1$ and $M_2$ be two consecutive moves applied to a state $\Lambda$, then these moves can be swapped whenever $M_2(\Lambda)$ and $M_1(M_2(\Lambda))$ make sense, and $M_1(M_2(\Lambda))=M_2(M_1(\Lambda))$. An adequate understanding of swaps in a sequence of moves is a necessary tool for the proof of the main result.

Section~\ref{Technical_Lemmas_2} contains technical results.  In Section~\ref{MainTheorem} we present the proof of the main result, Theorem~\ref{avoidingtype456}, that says that a multiset of integer numbers is $C_{\mathbb{Z},6}-$realizable if and only if it is  $C_{\mathbb{Z},3}-$realizable. Finally, in Section~\ref{Final_remarks} we discuss the extension of the main theorem to rationals and to reals. We also discuss the decision problem of deciding when a given a multiset is $C-$realizable. In this sense, Borobia and Canogar~\cite{bc2017} proved that the decision problem of determining if $\Lambda$ is $C-$realizable is NP-hard. This result should not discourage the search for more efficient algorithms.

\section{An alternative approach to $C-$realizability} \label{Another-approach-to-C}

 As we said in the introduction, it will be convenient to restate Theorem~\ref{guo_th} by partitioning it into four items.
\begin{theorem}\label{guo_th_2}
Let $\Lambda=\{\lambda_1,\ldots,\lambda_n\}$ be a realizable multiset of $\mathbb{C}$ whose  Perron root is $\lambda_1$ and let  $\lambda_2$ be real. For any $\epsilon>0$ we have that:
\begin{enumerate}    
    \item \label{move-c-1} If $\lambda_2\leq 0$, then the multiset $\{\lambda_1+ \epsilon,\lambda_2- \epsilon,\ldots,\lambda_n\}$ is realizable.
    \item \label{move-c-2} If $\lambda_2\geq \epsilon$, then the multiset $\{\lambda_1+ \epsilon,\lambda_2- \epsilon,\ldots,\lambda_n\}$ is realizable.
    \item \label{move-c-3} If $\lambda_2< 0$ with  $|\lambda_2|\geq \epsilon$, then the multiset $\{\lambda_1+ \epsilon,\lambda_2+ \epsilon,\ldots,\lambda_n\}$ is realizable.
    \item \label{move-c-4} If $\lambda_2\geq 0$, then the multiset $\{\lambda_1+ \epsilon,\lambda_2+ \epsilon,\ldots,\lambda_n\}$ is realizable.
\end{enumerate}
\end{theorem}

Let us see the equivalence of Theorem~\ref{guo_th} and Theorem~\ref{guo_th_2}. That Theorem~\ref{guo_th} implies Theorem~\ref{guo_th_2} is obvious. On the other hand, the cases of Theorem~\ref{guo_th} that are not covered by Theorem~\ref{guo_th_2} are:
\begin{itemize}
    \item $\{\lambda_1+ \epsilon,\lambda_2- \epsilon,\ldots,\lambda_n\}$ when $\epsilon>\lambda_2>0$. It can be obtained in two steps:
$$
\{\lambda_1,\lambda_2,\ldots,\lambda_n\} \longrightarrow 
\{\lambda_1+\epsilon_1, \lambda_2-\epsilon_1,\ldots,\lambda_n\}
\longrightarrow 
\{\lambda_1+ \epsilon_1+\epsilon_2,\lambda_2-\epsilon_1-\epsilon_2,\ldots,\lambda_n\}
$$
with $\epsilon_1=\lambda_2$ as in item~\ref{move-c-2} of Theorem~\ref{guo_th_2}, and  $\epsilon_2=\epsilon-\lambda_2$ as in item~\ref{move-c-1} of Theorem~\ref{guo_th_2}.

    \item $\{\lambda_1+ \epsilon,\lambda_2+ \epsilon,\ldots,\lambda_n\}$ when $\lambda_2<0$ and $\epsilon>|\lambda_2|$. It can be obtained in two steps:
$$
\{\lambda_1,\lambda_2,\ldots,\lambda_n\} \longrightarrow 
\{\lambda_1+\epsilon_1, \lambda_2+\epsilon_1,\ldots,\lambda_n\}
\longrightarrow 
\{\lambda_1+ \epsilon_1+\epsilon_2,\lambda_2+\epsilon_1+\epsilon_2,\ldots,\lambda_n\}
$$
with $\epsilon_1=|\lambda_2|$ as in item~\ref{move-c-3} of Theorem~\ref{guo_th_2}, and  $\epsilon_2=\epsilon-|\lambda_2|$ as in item~\ref{move-c-4} of Theorem~\ref{guo_th_2}.

\end{itemize}

Now we group together Theorem~\ref{join}, Theorem~\ref{1+} and Theorem~\ref{guo_th_2} particularizing them in the reals, the rationals or the integers.

\begin{theorem} \label{guo-etc}
 Let $\Lambda=\{\lambda_1,\ldots,\lambda_n\}$  and $\Gamma$ be two realizable multisets of $\mathbb{S}$ where $\mathbb{S}$ is $\mathbb{R}$, $\mathbb{Q}$ or $\mathbb{Z}$. Let  $\lambda_1$ be  the  Perron root of $\Lambda$ and let $\epsilon>0$ with $\epsilon\in \mathbb{S}$. Then
\begin{enumerate}[(i)]
    \item \label{move-i-guo} The multiset $\Lambda\cup \Gamma$ is realizable. 
    \item \label{move-ii-guo} If $\epsilon >0$, then the set $\{\lambda_1 +\epsilon,\lambda_2,\ldots,\lambda_n\}$ is realizable.
    \item \label{move-iii-guo} If $\lambda_2\leq 0$, then the multiset $\{\lambda_1+ \epsilon,\lambda_2- \epsilon,\lambda_3,\ldots,\lambda_n\}$ is realizable.
    \item \label{move-iv-guo} If $\lambda_2\geq \epsilon$, then the multiset $\{\lambda_1+ \epsilon,\lambda_2- \epsilon,\lambda_3,\ldots,\lambda_n\}$ is realizable.
    \item \label{move-v-guo} If $\lambda_2< 0$ with  $|\lambda_2|\geq \epsilon$, then the multiset $\{\lambda_1+ \epsilon,\lambda_2+ \epsilon,\lambda_3,\ldots,\lambda_n\}$ is realizable.
    \item \label{move-vi-guo} If $\lambda_2\geq 0$, then the multiset $\{\lambda_1+ \epsilon,\lambda_2+ \epsilon,\lambda_3,\ldots,\lambda_n\}$ is realizable.
\end{enumerate}
\end{theorem}

Now we return to the concept of $C-$realizability by employing Theorem~\ref{guo-etc}, and expanding the definition to reals, rationals and integers. 

\begin{definition} \label{CZ_realizability_6_rules}
Let $\{\lambda_1,\ldots,\lambda_n\}$ be  a multiset  of $\mathbb{S}$ where $\mathbb{S}$ is $\mathbb{R}$, $\mathbb{Q}$ or $\mathbb{Z}$. We say that $\Lambda$ is  {\bf $C_{\mathbb{S},6}-$realizable}
if it can be reached starting from the $n$ realizable basic  multisets $\{0\},  \ldots, \{0\}$ and successively applying any finite number of times any of the six rules~\eqref{move-i-guo}--\eqref{move-vi-guo} of  Theorem \ref{guo-etc}. 

\end{definition} 

Of course,  the  $C-$realizable multisets of  Definition~\ref{C_realizability_2} are the $C_{\mathbb{R},6}-$realizable multisets of Definition~\ref{CZ_realizability_6_rules}.

The partition of Theorem~\ref{guo_th} into the four items of Theorem~\ref{guo_th_2}  is closely related with the simplification that we pursue. Indeed we will see that when  $\mathbb{S}=\mathbb{Z}$, rules~\eqref{move-iv-guo}--\eqref{move-vi-guo} of Theorem~\ref{guo-etc} are redundant and can be omitted. 
For the cases when $\mathbb{S}$ is $\mathbb{Q}$ or $\mathbb{R}$ we were unable to prove the same. In Section~\ref{extension-rationals-reals} we will be more specific about the obstacle that we encounter with $\mathbb{S}=\mathbb{Q}$, while the case $\mathbb{S}=\mathbb{R}$ is much more elusive. That said, it is convenient to also have the following definition where the 6 rules are replaced by the 3 rules~\eqref{move-i-guo}--\eqref{move-iii-guo}.

\begin{definition} \label{CZ_realizability_2_rules}
Let  $\mathbb{S}$ be $\mathbb{R}$, $\mathbb{Q}$ or $\mathbb{Z}$. A multiset $\{\lambda_1,\ldots,\lambda_n\}$ of $\mathbb{S}$ is  {\bf $C_{\mathbb{S},3}-$realizable}
if it can be reached starting from the $n$ realizable multisets $\{0\}, \{0\}, \ldots, \{0\}$ and successively applying any finite number of times any of the three rules~\eqref{move-i-guo}--\eqref{move-iii-guo} of  Theorem \ref{guo-etc}.
\end{definition}

\subsection{$C-$realizability for integers}

In this work we will work mainly with  integer numbers. Note that  rules~\eqref{move-ii-guo}--\eqref{move-vi-guo} of Theorem~\ref{guo-etc} with an integer $\epsilon>1$, can be split into $\epsilon$ repetitions of the same rule  in which we replace $\epsilon$ by 1. So, when $\mathbb{S}=\mathbb{Z}$, Theorem~\ref{guo-etc} can be restated in a simplified way as follows. 

\begin{theorem} \label{guo-etc-2}
Let $\Lambda=\{\lambda_1,\ldots,\lambda_n\}$ and $\Gamma$ be two realizable multisets of $\mathbb{Z}$, and let $\lambda_1$ be the Perron root of $\Lambda$. Then
\begin{enumerate}[(i)]
    \item \label{move-i} The set $\Lambda\cup \Gamma$ is realizable.    \item \label{move-ii} The set $\{\lambda_1+1,\lambda_2,\ldots,\lambda_n\}$ is realizable.
    \item \label{move-iii} If $\lambda_2\leq 0$, then the set $\{\lambda_1+ 1,\lambda_2- 1,\lambda_3,\ldots,\lambda_n\}$ is realizable.
    \item \label{move-iv} If $\lambda_2>0$, then the set $\{\lambda_1+ 1,\lambda_2- 1,\lambda_3,\ldots,\lambda_n\}$ is realizable.
    \item \label{move-v} If $\lambda_2< 0$, then the set $\{\lambda_1+ 1,\lambda_2+ 1,\lambda_3,\ldots,\lambda_n\}$ is realizable.
    \item \label{move-vi} If $\lambda_2\geq 0$, then the set $\{\lambda_1+ 1,\lambda_2+ 1,\lambda_3,\ldots,\lambda_n\}$ is realizable.
\end{enumerate}
\end{theorem}

Now, using Theorem~\ref{guo-etc-2} instead of Theorem~\ref{guo-etc} we can redefine $C_{\mathbb{Z},6}$ and $C_{\mathbb{Z},3}-$realizable multisets given in Definition~\ref{CZ_realizability_6_rules} and~\ref{CZ_realizability_2_rules} respectively.

\begin{definition} \label{CZ_realizability_2}
A multiset $\{\lambda_1,\ldots,\lambda_n\}$ of integer numbers is said to be {\bf $C_{\mathbb{Z},6}-$realizable}
if it can be reached starting from the $n$ realizable sets $\{0\},  \ldots, \{0\}$ and successively applying any finite number of times any of the six rules~\eqref{move-i}--\eqref{move-vi} of  Theorem \ref{guo-etc-2}. And is {\bf $C_{\mathbb{Z},3}-$realizable} if we only apply the three rules~\eqref{move-i}--\eqref{move-iii}
\end{definition}

A consequence of Theorem~\ref{th_original} is the following.

\begin{corollary}  \label{For_Z}
    If a multiset $\{\lambda_1,\ldots,\lambda_n\}$ of $\mathbb{Z}$ is $C_{\mathbb{Z},6}$ or   $C_{\mathbb{Z},3}-$realizable, then it is realizable.
\end{corollary}

\section{States and types of moves for states} \label{types_moves}

In the previous section we used multisets. The elements of multisets do not have an order and, consequently, analyzing a long sequence of rules with multisets would be challenging. On the other hand, having an order on the elements and preserving that order after any rule, would permit us to easily track the changes that the rules perform. In this sense we will replace the use of multisets by tuples.  Actually, we will need to manage a collection of tuples of different sizes, or a tuple of tuples and we will call them states. We introduce all the terminology referring to states in the following definition.

\begin{definition} We consider the following sets:

\begin{itemize}

\item For any $n\in \mathbb{N}$ and any $p\in\{1,\ldots,n\}$ we will consider the set  given by
\[
\mathbb{Z}^{n,p}:=\{(\Lambda_{1}, \ldots, \Lambda_{p})  : \Lambda_i \in \mathbb{Z}^{n_i}  ; \,   n=n_1+\cdots+n_p\}.
\]
  Each element of $\mathbb{Z}^{n,p}$ will be  called a {\bf state}. So, a state of $\mathbb{Z}^{n,p}$ is a tuple of $p$ tuples which might be of different sizes and with $n$  integers overall. 
For  $\big((\lambda_{1}, \ldots, \lambda_{n})\big)\in \mathbb{Z}^{n,1}$ we will also use the notation $(\lambda_{1}, \ldots, \lambda_{n})\in \mathbb{Z}^{n}$.

\item   To the state
\[
\Phi=\Big((\lambda_{1}, \ldots, \lambda_{n_1}),(\lambda_{n_1+1}, \ldots, \lambda_{n_1+n_2}),\ldots, (\lambda_{n_1+\cdots+n_{p-1}+1},\ldots, \lambda_{n})\Big)\in \mathbb{Z}^{n,p}
\]
we  associate the $n-$tuple
\[\Phi^*=(\lambda_{1}, \ldots,\lambda_{n})\in \mathbb{Z}^n.\]
The {\bf $k^{th}$ position} of state $\Phi$ denotes the specific coordinate and tuple where $\lambda_k$ is located inside $\Phi$. So it makes sense to talk about the tuple to which the $k^{th}$ position belongs. And the value at the {\bf $k^{th}$ position} of  $\Phi$ will be    $\Phi[k]=\Phi^*[k]=\lambda_k$. 

 \item Let $\Lambda=(\lambda_{1}, \ldots,\lambda_{n} )\in \mathbb{Z}^n$. We will say that $\lambda_i$ is \textbf{dominant} in $\Lambda$ if  $\lambda_{i}\geq |\lambda_{j}|$ for all $j=1,\ldots, n$. 
 
 \item Let $\Phi\in \mathbb{Z}^{n,p}$. We will say that \textbf{the $i^{th}$ position is dominant} in $\Phi$ if  the value at the $i^{th}$ position of $\Phi$ is dominant in the tuple  where it is located. For instance, in the state  
$$
\Phi= \big((8,-2,8,1,-4),(-1,5,-6),(2,-3,5,-5)\big) \in \mathbb{Z}^{12,3}
$$
the first tuple of $\Phi$ has two dominant positions in the  $1^{st}$ and $3^{rd}$ positions;  the second tuple of $\Phi$ has no dominant position;  and the third tuple of $\Phi$ has one  dominant position in the $11^{th}$ position.

\end{itemize}
\end{definition}

The following definition translates the six rules of Theorem~\ref{guo-etc-2} applied to multisets into six  moves  applied to states. And these will be the terms we will use throughout this work, that is, rules for multisets and moves  for states. 

\begin{definition}
A \textbf{move} is a function that transforms a state into another state.  The first move that we consider transforms a state of $\mathbb{Z}^{n,p}$ into another state of $\mathbb{Z}^{n,p-1}$:
\begin{itemize}
\item For $1\leq d < p$, the {\bf type 1 move} $T_1^{(d)^\frown (d+1)}$ acts on $\Phi$ so that $T_1^{(d)^\frown (d+1)}(\Phi)$ is obtained from $\Phi$ by joining or concatenating the $d^{th}$  and  $(d+1)^{th}$ tuples of $\Phi$. We can only join two consecutive tuples, and there is no reordering of the values in the new bigger tuple, in other words $T_1^{(d)^\frown (d+1)}(\Phi)[j]=\Phi[j]$ for all $j=1,\ldots, n$. 

We will say that $T_1^{(d)^\frown (d+1)}$ is a \textbf{valid move for $\Phi$}.
\end{itemize}

And the rest of the moves  transform a state  $\Phi$ of $\mathbb{Z}^{n,p}$ into another state of $\mathbb{Z}^{n,p}$:

\begin{itemize}
\item For $1\leq i\leq n$, the {\bf type 2 move} $T^{i}_2$ acts on $\Phi$ by only modifying the $i^{th}$ position, namely,  $T^{i}_2(\Phi)[i]=\Phi[i]+1$. 

We will say that $T^{i}_2$ is a \textbf{valid move for $\Phi$} if its $i^{th}$ position is dominant, otherwise we will say that $T^i_2$ is an \textbf{invalid move for $\Phi$}.

\item For $1\leq i,j \leq n$ with $i\neq j$, the {\bf type 3 move} $T^{ij}_3$ acts by only modifying the $i^{th}$ and $j^{th}$ positions, namely,   $T^{i,j}_3(\Phi)[i]=\Phi[i]+1$ and $T^{i,j}_3(\Phi)[j]=\Phi[j]-1$.  

We will say that $T^{ij}_3$ is a \textbf{valid move for $\Phi$} if the $i^{th}$ and $j^{th}$ position in $\Phi$ are in the same tuple, the $i^{th}$ position is dominant in $\Phi$, and $\Phi[j]\leq 0$, otherwise we will say that $T^{ij}_3(\Phi)$ is an \textbf{invalid move for $\Phi$}.

\item For $1\leq i,j \leq n$ with $i\neq j$, the {\bf type 4 move} $T^{ij}_4$ acts by only modifying the $i^{th}$ and $j^{th}$ positions, namely,   $T^{i,j}_4(\Phi)[i]=\Phi[i]+1$ and $T^{i,j}_4(\Phi)[j]=\Phi[j]-1$.  

Note that $T^{ij}_4$ is equal to $T^{ij}_3$ as transformation, the difference lies on the validity or invalidity. We will say that $T^{ij}_4$ is a \textbf{valid move for $\Phi$} if  the $i^{th}$ and $j^{th}$ position in $\Phi$ are in the same tuple, the $i^{th}$ position is dominant in $\Phi$, and $\Phi[j]> 0$, otherwise we will say that $T^{ij}_4$ is an \textbf{invalid move for $\Phi$}.

\item For $1\leq i,j \leq n$ with $i\neq j$, the {\bf type 5 move} $T^{ij}_5$ acts by only modifying the $i^{th}$ and $j^{th}$ positions, namely,   $T^{ij}_5(\Phi)[i]=\Phi[i]+1$ and $T^{ij}_5(\Phi)[j]=\Phi[j]+1$. 

We will say that $T^{ij}_5$ is a \textbf{valid move for $\Phi$} if  the $i^{th}$ and $j^{th}$ position in $\Phi$ are in the same tuple, the $i^{th}$ position is dominant in $\Phi$, and $\Phi[j]< 0$, otherwise we will say that $T^{ij}_5$ is an \textbf{invalid move for $\Phi$}.

\item For $1\leq i,j \leq n$ with $i\neq j$, the {\bf type 6 move} $T^{ij}_6$ acts by only modifying the $i^{th}$ and $j^{th}$ positions, namely,   $T^{ij}_6(\Phi)[i]=\Phi[i]+1$ and $T^{ij}_6(\Phi)[j]=\Phi[j]+1$.  

Note, again,  that $T^{ij}_6$ is equal to $T^{ij}_5$ as transformation, the difference lies on the validity or invalidity. We will say that $T^{ij}_6$ is a \textbf{valid move for $\Phi$} if  the $i^{th}$ and $j^{th}$ position in $\Phi$ are in the same tuple, the $i^{th}$ position is dominant in $\Phi$, and $\Phi[j]\geq 0$, otherwise we will sat that $T^{ij}_6$ is an \textbf{invalid move for $\Phi$}.

\end{itemize}
\end{definition}

All these moves, valid or invalid, will be said to be {\bf admissible moves}. With this we want to emphasize that  they can be performed, independently of its validity. Moves that are not admissible are moves that can not be performed, we will give examples of them bellow.

To clarify the idea of valid and invalid moves, we provide examples for each of the six types of moves when they are applied to  state
\begin{equation*}\label{example1}
\Phi= \big((8,-2,-3,1,-4),(5,-1,-2)\big) \in \mathbb{Z}^{8,2}.
\end{equation*} 
We consider the following admissible moves:
\begin{itemize}
\item $T_1^{(1)\frown(2)}(\Phi)= \big((8,-2,-3,1,-4,5,-1,-2)\big)$; 
\item $T_2^6(\Phi)= \big((8,-2,-3,1,-4),(6,-1,-2)\big)$;
\item $T_2^8(\Phi)= \big((8,-2,-3,1,-4),(5,-1,-1)\big)$;
\item $T_3^{6,8}(\Phi)=T_4^{6,8}(\Phi)= \big((8,-2,-3,1,-4),(6,-1,-3)\big)$; 
\item $T_3^{1,4}(\Phi)=T_4^{1,4}(\Phi)= \big((9,-2,-3,0,-4),(5,-1,-2)\big)$;
\item $T_3^{1,8}(\Phi)=T_4^{1,8}(\Phi)= \big((9,-2,-3,1,-4),(5,-1,-3)\big)$;
\item  $T_5^{6,8}(\Phi)=T_6^{6,8}(\Phi)= \big((8,-2,-3,1,-4),(6,-1,-1)\big)$;
\item $T_5^{1,4}(\Phi)=T_6^{1,4}(\Phi)= \big((9,-2,-3,2,-4),(5,-1,-2)\big)$;
\item $T_5^{1,8}(\Phi)=T_6^{1,8}(\Phi)= \big((9,-2,-3,1,-4),(5,-1,-1)\big)$.

\end{itemize}

Note  that for state $\Phi$ the move $T_1^{(1)\frown(2)}$ is valid, $T_2^6$ is valid,  $T_2^8$ is invalid,  $T_3^{6,8}$ is  valid   while $T_4^{6,8}$ is  invalid,   $T_3^{1,4}$ is invalid while  $T_4^{1,4}$ is  valid,   $T_3^{1,8}$ and $T_4^{1,8}$ are invalid, $T_5^{6,8}$ is  valid   while $T_6^{6,8}$ is  invalid,   $T_5^{1,4}$ is  invalid  while $T_6^{1,4}$ is valid, and finally $T_5^{1,8}$ and $T_6^{1,8}$ are invalid.

Finally we provide some moves that are not admissible for $\Phi$: $T_1^{(2)\frown(3)}$, $T_2^{10}$, $T_3^{1,9}$, $T_4^{11,12}$, $T_5^{12,1}$, and $T_6^{1,13}$. In all of them, at least one of the super-indexes is out of bounds and nothing can be done with $\Phi$.

\subsection{Laws for moves} \label{Laws_for_moves}

Consider the sequence
\begin{align*}
    &\Phi_1 \xrightarrow{M_1} \Phi_2 \xrightarrow{M_2}  \cdots \rightarrow \Phi_{k-1} \xrightarrow{M_{k-1}}\Phi_{k}  \xrightarrow{M_{k}} \Phi_{k+1}
\end{align*}
where $\Phi_1, \ldots, \Phi_{k+1}$ are states and $M_1, \ldots, M_{k}$ are valid moves that verify $M_i(\Phi_i)=\Phi_{i+1}$ for $i=1, \ldots, k$.
Now we provide, with respect to this sequence and any state $\Phi_t$ in that sequence, the following  list of laws that will be assumed throughout the whole work.
The proof of each law is  easy to deduce:
\begin{enumerate}

\item  \textbf{Permanence law}: If the $i^{th}$ and $j^{th}$ positions are in the same tuple of $\Phi_t$, then both positions remain on the same tuple in any subsequent states.  This law also works when the moves are admissible, valid or invalid.

\item \textbf{Dominance law}: When only valid moves  are applied, if a position is  dominant in $\Phi_t$ then it was  dominant in all previous states and,
equivalently, if a position is not dominant in a state then it will never be dominant in subsequent states.

\item \textbf{Increasing-decreasing law}:  When only valid moves of types 1, 2, 3 and 6 are applied, the positive values can only increase, and the negative values can only decrease. 
\end{enumerate}


\section{$C^*_{\mathbb{Z},6}$ and $C^*_{\mathbb{Z},3}-$realizability for states} \label{C*6-realizability and C*3-realizability}

Note that the order in which the elements are allocated on the multiset $\{\lambda_1,\ldots,\lambda_n\}$ is not important when we determine its $C_{\mathbb{Z},6}$ or $C_{\mathbb{Z},3}-$realizability. On the other hand, on the next definition  applied to the tuple $(\lambda_1,\ldots,\lambda_n)$,  the order   is essential. 

\begin{definition}
The state $(\lambda_1,\ldots,\lambda_n)\in\mathbb{Z}^{n}$  is {\bf $C_{\mathbb{Z},6}^*-$realizable} if can be obtained from state $\big((0),\ldots, (0)\big)\in\mathbb{Z}^{n,n}$ by applying a sequence of  valid moves of types 1 to 6. 
And is {\bf $C_{\mathbb{Z},3}^*-$realizable} if we only apply valid moves of types 1 to 3. 
\end{definition}

Now we will see how $C_{\mathbb{Z},6}-$realizability is closely related to $C^*_{\mathbb{Z},6}-$realizability.

\begin{lemma} \label{Crealiz_C6realiz}
    A multiset $\{\lambda_1,\ldots,\lambda_n\}$ is  $C_{\mathbb{Z},6}-$realizable if and only if there is at least one permutation $\sigma$ of $n$ elements such that the state $(\lambda_{\sigma(1)},\ldots,\lambda_{\sigma(n)})\in\mathbb{Z}^{n}$  is $C^*_{\mathbb{Z},6}-$realizable. 
\end{lemma}

\begin{proof}

The sufficiency is quite straightforward. Suppose that we start with  $\big((0),\ldots, (0)\big)$ and, for some permutation $\sigma$,  we arrive to $(\lambda_{\sigma(1)},\ldots,\lambda_{\sigma(n)})$ by applying a sequence of  valid moves of types 1 to 6. If we change  tuples by multisets and  moves by rules, then we start with  $\{0\},\ldots, \{0\}$ and  apply the corresponding sequence of rules~\eqref{move-i}--\eqref{move-vi} of  Theorem \ref{guo-etc-2} to obtain the multiset $\{\lambda_{\sigma(1)},\ldots,\lambda_{\sigma(n)}\}$. And, of course, this multiset is equal to the multiset $\{\lambda_1,\ldots,\lambda_n\}$.

It will be easier to prove the necessity by showing the behaviour of one example, which we will analyze to understand how to find the permutation  of the statement of this lemma.  In the example we only consider rules of types~\eqref{move-i-guo} and~\eqref{move-iii-guo}, and at the end of the example we will explain why.

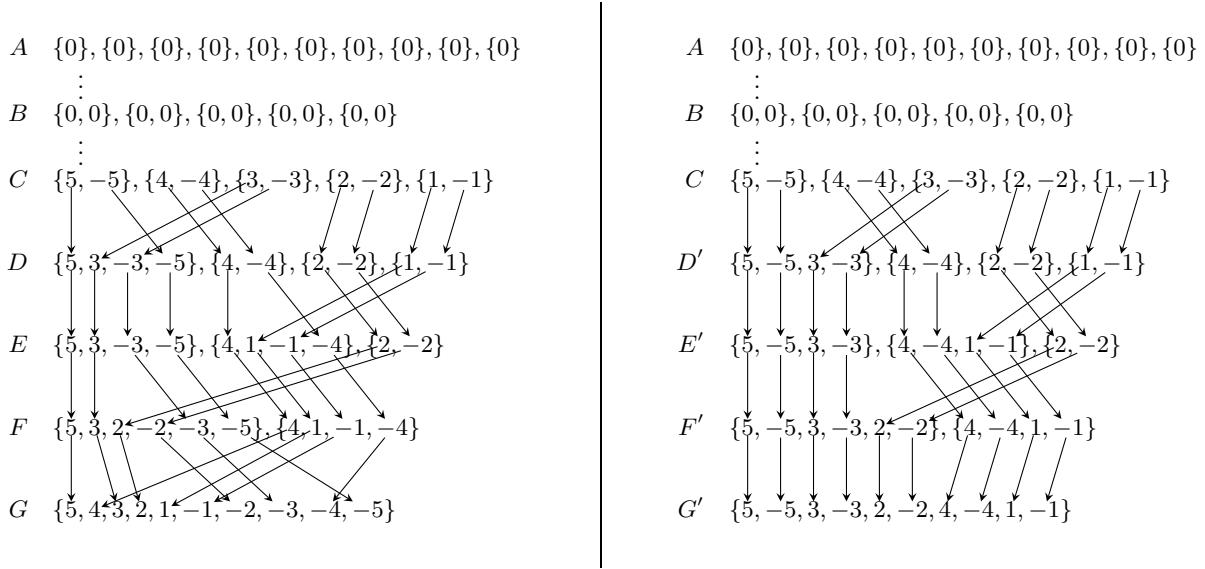
\begin{figure}[h!]
\begin{parcolumns}[rulebetween=true]{2}
    \colchunk[1]{
{\small
\begin{align*}
A&\quad \{0\},\{0\},\{0\},\{0\},\{0\},\{0\},\{0\},\{0\},\{0\},\{0\} \\[-2mm]
 & \qquad \vdots  \\[-2mm]
B&\quad \{0,0\},\{0,0\},\{0,0\},\{0,0\},\{0,0\} \\[-2mm]
 & \qquad \vdots  \\[-2mm]
C&\quad \{\tikznode{a0}{$5$},\tikznode{a1}{$-5$}\},\{\tikznode{a2}{$4$},\tikznode{a3}{$-4$}\},\{\tikznode{a4}{$3$},\tikznode{a5}{$-3$}\},\{\tikznode{a6}{$2$},\tikznode{a7}{$-2$}\},\{\tikznode{a8}{$1$},\tikznode{a9}{$-1$}\}\\[6mm]
D&\quad \{\tikznode{b'0}{$5$},\tikznode{b'4}{$3$},\tikznode{b'5}{$-3$}, \tikznode{b'1}{$-5$}\} ,\{\tikznode{b'2}{$4$},\tikznode{b'3}{$-4$}\}, \{\tikznode{b'6}{$2$},\tikznode{b'7}{$-2$}\},\{\tikznode{b'8}{$1$},\tikznode{b'9}{$-1$}\}\\[6mm]
E&\quad \{\tikznode{c'0}{$5$},\tikznode{c'4}{$3$},\tikznode{c'5}{$-3$},\tikznode{c'1}{$-5$}\} , \{\tikznode{c'2}{$4$},\tikznode{c'8}{$1$},\tikznode{c'9}{$-1$} ,\tikznode{c'3}{$-4$}\},\{\tikznode{c'6}{$2$},\tikznode{c'7}{$-2$}\}\\[6mm]
F&\quad \{\tikznode{d'0}{$5$},\tikznode{d'4}{$3$}, \tikznode{d'6}{$2$},\tikznode{d'7}{$-2$},\tikznode{d'5}{$-3$}, \tikznode{d'1}{$-5$}\}, \{\tikznode{d'2}{$4$},\tikznode{d'8}{$1$},\tikznode{d'9}{$-1$}  ,\tikznode{d'3}{$-4$}\}\\[6mm]
G&\quad \{\tikznode{e'0}{$5$}, \tikznode{e'2}{$4$}, \tikznode{e'4}{$3$}, \tikznode{e'6}{$2$},\tikznode{e'8}{$1$}, \tikznode{e'9}{$-1$}, \tikznode{e'7}{$-2$}, \tikznode{e'5}{$-3$}, \tikznode{e'3}{$-4$},\tikznode{e'1}{$-5$}\}   
\end{align*}}
\begin{tikzpicture}[remember picture,overlay
,rounded corners]
\path[->] (a0) edge (b'0); 
\path[->] (a1) edge (b'1); 
\path[->] (a2) edge (b'2); 
\path[->] (a3) edge (b'3); 
\path[->] (a4) edge (b'4); 
\path[->] (a5) edge (b'5); 
\path[->] (a6) edge (b'6); 
\path[->] (a7) edge (b'7); 
\path[->] (a8) edge (b'8); 
\path[->] (a9) edge (b'9);
\path[->] (b'0) edge (c'0); 
\path[->] (b'1) edge (c'1); 
\path[->] (b'2) edge (c'2); 
\path[->] (b'3) edge (c'3); 
\path[->] (b'4) edge (c'4); 
\path[->] (b'5) edge (c'5); 
\path[->] (b'6) edge (c'6); 
\path[->] (b'7) edge (c'7); 
\path[->] (b'8) edge (c'8); 
\path[->] (b'9) edge (c'9);
\path[->] (c'0) edge (d'0); 
\path[->] (c'1) edge (d'1); 
\path[->] (c'2) edge (d'2); 
\path[->] (c'3) edge (d'3); 
\path[->] (c'4) edge (d'4); 
\path[->] (c'5) edge (d'5); 
\path[->] (c'6) edge (d'6); 
\path[->] (c'7) edge (d'7); 
\path[->] (c'8) edge (d'8); 
\path[->] (c'9) edge (d'9);
\path[->] (d'0) edge (e'0); 
\path[->] (d'1) edge (e'1); 
\path[->] (d'2) edge (e'2); 
\path[->] (d'3) edge (e'3); 
\path[->] (d'4) edge (e'4); 
\path[->] (d'5) edge (e'5); 
\path[->] (d'6) edge (e'6); 
\path[->] (d'7) edge (e'7); 
\path[->] (d'8) edge (e'8); 
\path[->] (d'9) edge (e'9);
\end{tikzpicture} }
\colchunk[2]{
{\small
\begin{align*}
A&\quad \{0\},\{0\},\{0\},\{0\},\{0\},\{0\},\{0\},\{0\},\{0\},\{0\} \\[-2mm]
 & \qquad \vdots  \\[-2mm]
B&\quad \{0,0\},\{0,0\},\{0,0\},\{0,0\},\{0,0\} \\[-2mm]
 & \qquad \vdots  \\[-2mm]
C&\quad \{\tikznode{A0}{$5$},\tikznode{A1}{$-5$}\},\{\tikznode{A2}{$4$},\tikznode{A3}{$-4$}\},\{\tikznode{A4}{$3$},\tikznode{A5}{$-3$}\},\{\tikznode{A6}{$2$},\tikznode{A7}{$-2$}\},\{\tikznode{A8}{$1$},\tikznode{A9}{$-1$}\}\\[6mm]
D'&\quad \{\tikznode{B0}{$5$},\tikznode{B1}{$-5$},\tikznode{B4}{$3$},\tikznode{B5}{$-3$}\}, \{\tikznode{B2}{$4$},\tikznode{B3}{$-4$}\}, \{\tikznode{B6}{$2$},\tikznode{B7}{$-2$}\},\{\tikznode{B8}{$1$},\tikznode{B9}{$-1$}\}\\[6mm]
E'&\quad \{\tikznode{C0}{$5$},\tikznode{C1}{$-5$} ,\tikznode{C4}{$3$},\tikznode{C5}{$-3$}\} , \{\tikznode{C2}{$4$},\tikznode{C3}{$-4$}, \tikznode{C8}{$1$},\tikznode{C9}{$-1$}\},\{\tikznode{C6}{$2$},\tikznode{C7}{$-2$}\}\\[6mm]
F'&\quad \{\tikznode{D0}{$5$}, \tikznode{D1}{$-5$},\tikznode{D4}{$3$},\tikznode{D5}{$-3$}, \tikznode{D6}{$2$},\tikznode{D7}{$-2$}\}, \{\tikznode{D2}{$4$},\tikznode{D3}{$-4$},\tikznode{D8}{$1$},\tikznode{D9}{$-1$} \}\\[6mm]
G'&\quad \{\tikznode{E0}{$5$}, \tikznode{E1}{$-5$},\tikznode{E4}{$3$},\tikznode{E5}{$-3$}, \tikznode{E6}{$2$},\tikznode{E7}{$-2$},\tikznode{E2}{$4$},\tikznode{E3}{$-4$},\tikznode{E8}{$1$},\tikznode{E9}{$-1$} \}       
\end{align*}}
\begin{tikzpicture}[remember picture,overlay,
rounded corners]
\path[->] (A0) edge (B0); 
\path[->] (A1) edge (B1); 
\path[->] (A2) edge (B2); 
\path[->] (A3) edge (B3); 
\path[->] (A4) edge (B4); 
\path[->] (A5) edge (B5); 
\path[->] (A6) edge (B6); 
\path[->] (A7) edge (B7); 
\path[->] (A8) edge (B8); 
\path[->] (A9) edge (B9);
\path[->] (B0) edge (C0); 
\path[->] (B1) edge (C1); 
\path[->] (B2) edge (C2); 
\path[->] (B3) edge (C3); 
\path[->] (B4) edge (C4); 
\path[->] (B5) edge (C5); 
\path[->] (B6) edge (C6); 
\path[->] (B7) edge (C7); 
\path[->] (B8) edge (C8); 
\path[->] (B9) edge (C9);
\path[->] (C0) edge (D0); 
\path[->] (C1) edge (D1); 
\path[->] (C2) edge (D2); 
\path[->] (C3) edge (D3); 
\path[->] (C4) edge (D4); 
\path[->] (C5) edge (D5); 
\path[->] (C6) edge (D6); 
\path[->] (C7) edge (D7); 
\path[->] (C8) edge (D8); 
\path[->] (C9) edge (D9);
\path[->] (D0) edge (E0); 
\path[->] (D1) edge (E1); 
\path[->] (D2) edge (E2); 
\path[->] (D3) edge (E3); 
\path[->] (D4) edge (E4); 
\path[->] (D5) edge (E5); 
\path[->] (D6) edge (E6); 
\path[->] (D7) edge (E7); 
\path[->] (D8) edge (E8); 
\path[->] (D9) edge (E9);
\end{tikzpicture} }
\end{parcolumns}
\caption{(left) ordering the entries on each multiset, (right) without ordering the entries on each multiset}
\label{ordering-entries-multisets}
\end{figure} 

Let us see that $\{5,4,3,2,1,-1,-2,-3,-4,-5\}$ is $C_{\mathbb{Z},6}-$realizable. Consider the sequence of moves given on the left of Figure~\ref{ordering-entries-multisets}
where the step from $A$ to $B$ is by applying five times rule~\eqref{move-i} of Theorem~\ref{guo-etc-2}, and the second step from $B$ to $C$ is by applying fifteen times rule~\eqref{move-iii} of Theorem~\ref{guo-etc-2}. These two steps can be performed without the interchange of the position of the values. From $C$ to $G$ we have  interchanges of the position of some values as indicated by the arrows. 
In each multiset the values are ordered in non-decreasing order. The first insight after following all the threads, is that sorting at every step adds complexity. So on the right of Figure~\ref{ordering-entries-multisets} we leave the multisets unsorted which minimizes the number of crossings.

Once we have unsorted multisets, the structure of applying rule~\eqref{move-i} repeatedly 
is that of a tangled tree, in which the root of the tree, $G'$, is the only multisets at the bottom (see left of Figure~\ref{tangled-untangled-tree}). As we have a tree, it is  possible to untangle the threads so that there are no crossings (see right of Figure~\ref{tangled-untangled-tree}).

\begin{figure}[h!] 
\begin{parcolumns}[rulebetween=true]{2}
    \colchunk[1]{
{\small\begin{align*}
C&\quad \tikznode{P5}{$\{5,-5\}$},\tikznode{P4}{$\{4,-4\}$}, \tikznode{P3}{$\{3,-3\}$},\tikznode{P2}{$\{2,-2\}$}, \tikznode{P1}{$\{1,-1\}$}\\[6mm]
D'&\quad \tikznode{Q53}{$\{5,-5,3,-3\}$}, \tikznode{Q4}{$\{4,-4\}$},\tikznode{Q2}{$\{2,-2\}$}, \tikznode{Q1}{$\{1,-1\}$} \\[6mm]
E'&\quad \tikznode{R53}{$\{5,-5,3,-3\}$} \tikznode{R41}{$\{4,-4,1,-1\}$}, \tikznode{R2}{$\{2,-2\}$} \\[6mm]
F'&\quad \tikznode{S532}{$\{5,-5,3,-3,2,-2\}$}, \tikznode{S41}{$\{4,-4,1,-1\}$}\\[4mm]
G'&\quad \tikznode{T53241}{$\{5,-5,3,-3,2,-2,4,-4,1,-1\}$}       
\end{align*}}
\begin{tikzpicture}[remember picture,overlay,
rounded corners]
\path[->] (P1) edge (Q1); 
\path[->] (P2) edge (Q2); 
\path[->] (P3) edge (Q53); 
\path[->] (P4) edge (Q4); 
\path[->] (P5) edge (Q53); 
\path[->] (Q53) edge (R53);
\path[->] (Q2) edge (R2);
\path[->] (Q4) edge (R41);
\path[->] (Q1) edge (R41);
\path[->] (R53) edge (S532);
\path[->] (R2) edge (S532);
\path[->] (R41) edge (S41);
\path[->] (S532) edge (T53241);
\path[->] (S41) edge (T53241);
\end{tikzpicture} }
\colchunk[2]{
{\small\begin{align*}
C''&\quad \tikznode{p5}{$\{5,-5\}$}, \tikznode{p3}{$\{3,-3\}$},\tikznode{p2}{$\{2,-2\}$},\tikznode{p4}{$\{4,-4\}$}, \tikznode{p1}{$\{1,-1\}$}\\[6mm]
D''&\quad \tikznode{q53}{$\{5,-5,3,-3\}$}, \tikznode{q2}{$\{2,-2\}$}, \tikznode{q4}{$\{4,-4\}$}, \tikznode{q1}{$\{1,-1\}$} \\[6mm]
E''&\quad \tikznode{r53}{$\{5,-5,3,-3\}$},\tikznode{r2}{$\{2,-2\}$}, \tikznode{r41}{$\{4,-4,1,-1\}$} \\[6mm]
F'&\quad \tikznode{s532}{$\{5,-5,3,-3,2,-2\}$}, \tikznode{s41}{$\{4,-4,1,-1\}$}\\[4mm]
G'&\quad \tikznode{t53241}{$\{5,-5,3,-3,2,-2,4,-4,1,-1\}$}       
\end{align*}}
\begin{tikzpicture}[remember picture,overlay,
rounded corners]
\path[->] (p1) edge (q1); 
\path[->] (p2) edge (q2); 
\path[->] (p3) edge (q53); 
\path[->] (p4) edge (q4); 
\path[->] (p5) edge (q53); 
\path[->] (q53) edge (r53);
\path[->] (q2) edge (r2);
\path[->] (q4) edge (r41);
\path[->] (q1) edge (r41);
\path[->] (r53) edge (s532);
\path[->] (r2) edge (s532);
\path[->] (r41) edge (s41);
\path[->] (s532) edge (t53241);
\path[->] (s41) edge (t53241);
\end{tikzpicture}
 }
\end{parcolumns}
\caption{(left) tangled tree, (right) untangled tree.}
\label{tangled-untangled-tree}
\end{figure}
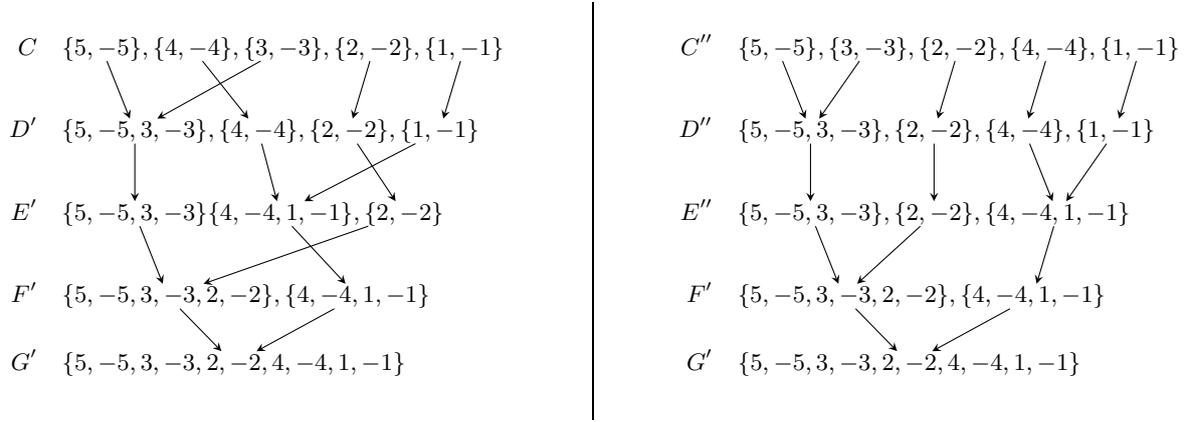    


On the right of Figure~\ref{tangled-untangled-tree} we observe that the position of the entries does not change across the different steps. This allow us to change the multisets to ordered list directly as follows:
\begin{align} 
\nonumber \Phi_{A}=& ((0),(0),(0),(0),(0),(0),(0),(0),(0),(0) )\\
\nonumber \Phi_{B}=& ((0,0),(0,0),(0,0),(0,0),(0,0) )\\
\nonumber \Phi_{C''}=& ((5,-5), (3,-3), (2,-2), (4,-4), (1,-1))\\
\label{Ex_tuples} \Phi_{D''}=& ((5,-5,3,-3), (2,-2), (4,-4), (1,-1))\\
\nonumber \Phi_{E''}=& ((5,-5,3,-3), (2,-2), (4,-4,1,-1)\\
\nonumber \Phi_{F'}=& ((5,-5,3,-3,2,-2),(4,-4,1,-1))\\
\nonumber \Phi_{G'}=& (5,-5,3,-3,2,-2,4,-4,1,-1)         
\end{align}


In summary, the aim of this process is to go from something without order (multisets) to something ordered (states). While everything regarding multisets (the collection of multisets and the values on each multiset) has no imposed order, everything regarding states (the collection of tuples and the values on each tuple) have an imposed order. 

In this proof, the transition between multisets and states is done in several stages. The first stage is to realize that presenting the values in non-decreasing order adds complexity, so we stop doing it (this has been done in Figure~\ref{ordering-entries-multisets}). The second stage is to find a correct way to order the multisets so that whenever a rule~\eqref{move-i} is applied then it joins two contiguous tuples (this reordering is possible and has been done in Figure~\ref{tangled-untangled-tree}). On the third stage we pass from multisets  (as on the right of Figure~\ref{tangled-untangled-tree}) to  states (as in~\eqref{Ex_tuples}). Observe that in the sequence of states each entry remains in the same position from the initial state to the final state.

A final remark, rules~\eqref{move-ii}--\eqref{move-vi} of Theorem~\ref{guo-etc-2} give no problems in the overall  process because they only modify some specific entries which do not require a change of positions. 
\end{proof}

Note that the state $\Lambda=(1,-1,1,-1)$ is  $C^*_{\mathbb{Z},6}-$realizable since $(1,-1)$ and $(1,-1)$ are both $C^*_{\mathbb{Z},6}-$realizable and can be joined by a type 1 move. But let us show that changing  the order of the entries we can obtain the state $\Lambda'=(1,1,-1,-1)$ which is not $C^*_{\mathbb{Z},6}-$realizable. Suppose that we could achieve $\Lambda'$ starting from $((0),(0),(0),(0))$ by a sequence of moves of type 1 to 6. First note that all the elements of $\Lambda'$ add to 0 so there are no moves of type 2, 5 or 6. And since $\Lambda'$ has two dominant entries equal to 1 then the last move is not of type 3 or 4. So the last move is necessarily of type 1: (a) joining  $(1)$ with $(1,-1,-1)$; (b) joining  $(1,1)$ with $(-1,-1)$; or (c) joining $(1,1,-1)$ with $(-1)$. In all cases the second tuple is not $C^*_{\mathbb{Z},6}-$realizable because the sum of its elements is negative.

So this is a case of two states with the same numbers, where one can be reached by a sequence of moves of types 1 to 6, while the other one can not be reached. This happens because none of the moves of type 1 to 6 perform a reordering of the elements of the tuples.

\medskip

 As in Lemma~\ref{Crealiz_C6realiz}, it is possible to  see how  $C_{\mathbb{Z},3}-$realizability  is closely related to  $C^*_{\mathbb{Z},3}-$realizability.
\begin{lemma} \label{Crealiz_C3realiz}
    A multiset $\{\lambda_1,\ldots,\lambda_n\}$ is $C_{\mathbb{Z},3}-$realizable if and only if there is at least one permutation $\sigma$ of $n$ elements such that the state $(\lambda_{\sigma(1)},\ldots,\lambda_{\sigma(n)})\in\mathbb{Z}^{n}$  is $C^*_{\mathbb{Z},3}-$realizable.  
\end{lemma}

Lemmas~\ref{Crealiz_C6realiz} and~\ref{Crealiz_C3realiz} permit us to translate the language of $C_{\mathbb{Z},6}$ and $C_{\mathbb{Z},3}-$realizable multisets into the language of $C^*_{\mathbb{Z},6}$ and $C^*_{\mathbb{Z},3}-$realizable states. They  will be useful in the proof of Theorem~\ref{avoidingtype456}.


\section{Swapping moves} \label{Technical_Lemmas}

From now on, in all sequences of admissible moves we will employ the following notation:
\begin{itemize}
    \item If the move is inside a  circle 
        $$
    \Phi  \xrightarrow{\text{\circled{$M$}}} M(\Phi)
    $$
then $M$ is admissible and invalid for $\Phi$. For instance $(1,2,7)  \xrightarrow{\text{\circled{$T^{2}_2$}}} (1,3,7)$. 

    \item If the move  has no circle around it
        $$\Phi  \xrightarrow{M} M(\Phi)$$
then $M$ is admissible and valid for $\Phi$. For instance $(1,2,7)  \xrightarrow{T^{3}_2} (1,2,8)$. 
\item If the move is inside a dashed circle
    $$
    \Phi  \xrightarrow{\text{\circledD{$M$}}} M(\Phi)
    $$
    then $M$ is admissible for $\Phi$, although we do not known yet if it is valid or invalid. This notation will be useful for some proofs in which the validity or invalidity of $M$ has to be determined.
\end{itemize}

Let   $\Phi$ be a state and let $M_1$ and $M_2$ be two moves of types 1 to 6 such that $M_1$ is admissible (valid or invalid) for $\Phi$ and $M_2$ is admissible (valid or invalid) for $M_1(\Phi)$. In order to prove the main result of this work, Theorem~\ref{avoidingtype456}, we are interested in understanding the effect on admissibility and validity when we swap the order in which  $M_1$ and $M_2$ act on $\Phi$.

The first result that we will see is that if in a sequence  of two admissible moves, valid or invalid, we perform a swap then in  most cases we obtain  again a sequence  of two admissible moves. 

\begin{lemma}\label{well_defined_swap}
   Let $\Phi\in \mathbb{Z}^{n,p}$ and consider a sequence  of admissible (valid or invalid) moves 
   \begin{align}  \label{seq_adm}
    &\Phi \xrightarrow{\text{\circledD{$M_1$}}} M_1(\Phi) \xrightarrow{\text{\circledD{$M_2$}}} M_2(M_1(\Phi)).
\end{align}
If  we   swap  $M_1$ and $M_2$  then we will obtain, except when   $M_1=T_1^{(p-1)^\frown (p)}$ and $M_2$ is of type 1, a sequence   
   \begin{align}  \label{seq_adm_swap}
    &\Phi \xrightarrow{\text{\circledD{$M_2$}}} M_2(\Phi) \xrightarrow{\text{\circledD{$M_1$}}} M_1(M_2(\Phi))
\end{align}
of admissible (valid or invalid) moves. Moreover, if $M_1$ and $M_2$ are not both of type 1 then $$M_1(M_2(\Phi))=M_2(M_1(\Phi)).$$
\end{lemma}

\begin{proof}  Note that a move $M$ of type 2 to 6 applied to a state of $n$ positions is admissible if and only if $M=T_2^i$ for some $1\leq i \leq n$ or $M=T_h^{jk}$ for some $3\leq h \leq 6$ and some $j\neq k$ with $1\leq j, k \leq n$. As the state $\Phi$ has $n$ positions, then all the states of  sequences (\ref{seq_adm}) and (\ref{seq_adm_swap}) will have also $n$ positions. Therefore, an admissible move of type 2 to 6 in sequence (\ref{seq_adm}) will remain admissible in sequence (\ref{seq_adm_swap}) after a swap.

Let us see what happens to $M_1$ and to $M_2$ when at least one of them is of type 1:

\begin{itemize}
    \item $M_1$ is a move of type 2 to 6 and $M_2=T_1^{(d)^\frown (d+1)}$ with $1\leq d< p$. After the swap $M_2$ is admissible for $\Phi$ since $\Phi$  has $p$ tuples.

    \item $M_1=T_1^{(d)^\frown (d+1)}$ with $1\leq d< p$ and $M_2$ is a move of type 2 to 6. After the swap $M_1$ is admissible for $M_2(\Phi)$  since $M_2(\Phi)$  has $p$ tuples.

    \item $M_1=T_1^{(d)^\frown (d+1)}$ with $1\leq d< p-1$ and $M_2=T_1^{(g)^\frown (g+1)}$ with $1\leq g< p-1$. After the swap $M_2$ is admissible for $\Phi$ since $\Phi$  has $p$ tuples, and  $M_1$ is  admissible for $M_2(\Phi)$ since $M_2(\Phi)$ has $p-1$ tuples.
    
    \item $M_1=T_1^{(p-1)^\frown (p)}$ and $M_2=T_1^{(d)^\frown (d+1)}$ with $1\leq d < p-1$. After the swap $M_2$ is admissible for $\Phi$ since $\Phi$  has $p$ tuples,  but $M_1$ becomes not admissible for $M_2(\Phi)$ since  $M_2(\Phi)$  has  $p-1$ tuples. 

\end{itemize}

Finally, let us see that if $M_1$ and $M_2$ are not both of type 1 then $M_1(M_2(\Phi))=M_2(M_1(\Phi))$. Consider two possibilities:

\begin{itemize}
    \item If one of the moves is of type 1, then both moves do not interfere with each other since one of them will modify one or two  positions and the other one will join two tuples. So the order in which they are applied does not matter. 
    \item If none of the moves is of type 1, then each move  will modify one or two positions of the state. And the the order in which the modifications are performed does not vary the final result. 
\end{itemize} 
\end{proof}

We provide one example of swapping two admissible moves  $M_1$ and $M_2$ of type 1:
$$ 
    ((0),(1),(2),(3)) \xrightarrow{T_1^{(1)^\frown (2)}} ((0,1),(2),(3)) \xrightarrow{T_1^{(2)^\frown (3)}} ((0,1),(2,3))
$$
$$((0),(1),(2),(3)) \xrightarrow{T_1^{(2)^\frown (3)}} ((0),(1,2),(3)) \xrightarrow{T_1^{(1)^\frown (2)}} ((0,1,2),(3)),
$$ 
and note that $M_1(M_2(\Phi))\neq M_2(M_1(\Phi))$. In the proof of the main theorem, we will never swap two admissible moves of Type 1, so we will never encounter this \emph{non-commutativity}.
\medskip

\subsection{Swapping a valid move of type 2}

Let us see how in a sequence of valid moves of types 1 to 3, we can swap  each type 2 move to bring it closer to the beginning of the sequence.

\begin{lemma} \label{swapT2}
Let $\Phi_1\in \mathbb{Z}^{n,p}$ and let us have a sequence of  valid moves
\begin{align}  \label{M_and_T2}
    &\Phi_1 \xrightarrow{M} \Phi_2 \xrightarrow{T^{i}_2} \Phi_{3},
\end{align}
where $M$ is of type 1, 2 or 3. If we swap $M$ and $T^{i}_2$ we obtain  the sequence of two valid moves 
\begin{align} \label{swap_two_moves}
   &\Phi_1 \xrightarrow{T^{i}_2} \Phi'_2 \xrightarrow{M} \Phi'_{3}
\end{align}
with $\Phi'_3=\Phi_3$.
\end{lemma}
\begin{proof}
As $\Phi_{2} \xrightarrow{T^{i}_2} \Phi_{3}$ is valid, then the $i^{th}$ position is  dominant in $\Phi_2$. By the dominance law the $i^{th}$ position is also dominant in $\Phi_1$, and therefore $\Phi_{1} \xrightarrow{T^{i}_2} \Phi'_{2}$ is valid. It remains to prove that $\Phi'_{2} \xrightarrow{M} \Phi'_{3}$ is also valid:
\begin{enumerate}
    \item  Let $M$ be a move of type 1. Then $\Phi'_{2} \xrightarrow{M} \Phi'_{3}$ is  admissible by Lemma~\ref{well_defined_swap}. So it is valid automatically.
    \item Let $M$  be a move $T_2^j$ of type 2.
    \begin{enumerate}[(a)]
        \item   If  $i=j$, then both moves are equal, so they can be swapped and remain valid.
        \item If $i\neq j$ and the $i^{th}$ and  $j^{th}$ positions  belong to   different tuples in $\Phi_1$, then  the two moves of sequence~\eqref{M_and_T2} are independent. So they can be swapped and remain valid.
        \item   Let us see that it is not possible for $i\neq j$ that the $i^{th}$ and the $j^{th}$ position  belong to the same tuple in $\Phi_1$. Since $\Phi_{1} \xrightarrow{T^{j}_2} \Phi_{2}$ is valid, then the $j^{th}$ position is strictly greater than any other value in its tuple of $\Phi_2$. This makes  $T^{i}_2$ invalid for $\Phi_2$, which contradicts~\eqref{M_and_T2}.
        

    \end{enumerate}
    
    \item  Let $M$ be a move $T_3^{jk}$ of type 3.
    \begin{enumerate}[(a)]
        \item Let $i=j$. As $\Phi_{1} \xrightarrow{T^{i}_2} \Phi'_{2}$ is valid, then the $i^{th}$ position is dominant in $\Phi'_2$ and so   $\Phi'_{2} \xrightarrow{T_3^{ik}} \Phi'_{3}$ is valid.
        \item If $i\neq j$ and the $i^{th}$ and  $j^{th}$ positions  belong to   different tuples in $\Phi_1$, then  the two moves of sequence~\eqref{M_and_T2} are independent. So they can be swapped and remain valid.
        \item   Let us see that it is not possible for $i\neq j$ that the $i^{th}$ and the $j^{th}$ position  belong to the same tuple in $\Phi_1$. Since $\Phi_{1} \xrightarrow{T^{jk}_3} \Phi_{2}$ is valid, then the $j^{th}$ position is strictly greater than any other value in its tuple of $\Phi_2$. This makes  $T^{i}_2$ invalid for $\Phi_2$, which contradicts~\eqref{M_and_T2}.

    \end{enumerate}
\end{enumerate}
\end{proof}

The counterpart of Lemma~\ref{swapT2} is false as we will see in the following example. We begin with the following sequence of valid moves:
$$\big((1,2),(7)\big) \ \xrightarrow{T^{2}_2} \ \big((1,3),(7)\big) \ \xrightarrow{T^{{(1)^\frown(2)}}_1}\ \big(1,3,7\big).$$
After the swap, the move of type 2 becomes invalid:
$$\quad \big((1,2),(7)\big) \ \xrightarrow{T^{{(1)^\frown(2)}}_1} \ \big(1,2,7\big) \ \xrightarrow{\text{\circled{$T^{2}_2$}}} \big(1,3, 7\big).$$

An interesting consequence of Lemma~\ref{swapT2} is that for a  $C^*_{\mathbb{Z},3}-$realizable state  all type 2 moves can be performed at the beginning of the sequence. 

\begin{corollary}\label{T2_begining}
    If $\Phi$ is a $C^*_{\mathbb{Z},3}-$realizable state, then $\Phi$ can be obtained by a sequence of moves of type 1 to 3, where all type 2 moves are performed at the beginning of the sequence.
\end{corollary}
\begin{proof}
    If $\Phi$ is a $C^*_{\mathbb{Z},3}-$realizable state, then $\Phi$ can be obtained by a sequence of moves of types 1 to 3. Applying repeatedly  Lemma~\ref{swapT2} we can take all type 2 moves to the beginning of the sequence.
\end{proof}
We will talk more about the implications of Corollary~\ref{T2_begining} in the final remarks.

\subsection{Swapping an invalid move of type 2}

The most frequent  swaps that we will perform in the proof of the main result (Theorem~\ref{avoidingtype456}) involve first a valid move of type 1, 2 or 3 followed by an invalid move of type 2. In the next three lemmas we analyze the result of these swaps.

\begin{lemma}\label{T1validT2invalid}
Let $\Phi_1\in \mathbb{Z}^{n,p}$  and let  
\begin{align*} 
    &\Phi_1 \xrightarrow{T_1^{(d)^\frown (d+1)}} \Phi_2 \xrightarrow{\text{\circled{$T^{i}_2$}}} \Phi_3
\end{align*} 
be a sequence where $T_1^{(d)^\frown (d+1)}$ is valid and $T^{i}_2$ is invalid. If we  swap both moves then we obtain the sequence 
\begin{align*} 
    &\Phi_1 \xrightarrow{\text{\circledD{$T^{i}_2$}}} T_2^i(\Phi_1) \xrightarrow{T_1^{(d)^\frown (d+1)}} \Phi_3
\end{align*} 
where $T_1^{(d)^\frown (d+1)}$  remains valid for $T_2^i(\Phi_1)$ and $T_2^i$ can be valid or invalid.
    
\end{lemma}
\begin{proof}
    As $T_2^i$ does not change the tuples of $\Phi_1$ then $T_1^{(d)^\frown (d+1)}$ will remain valid for $T_2^i(\Phi_1)$. On the other hand, as $T_2^i$ is invalid then the $i^{th}$  position is not dominant in $\Phi_2$. Since the $i^{th}$  position  in $\Phi_1$ lies in a  equal or smaller tuple than in $\Phi_2$, then we have two possibilities: if the $i^{th}$  position  in $\Phi_1$ is dominant then $T_2^i$ will become valid for $\Phi_1$, and if the $i^{th}$  position  in $\Phi_1$ remains not dominant then $T_2^i$ will become invalid for $\Phi_1$.
\end{proof}

\begin{lemma} \label{T2validT2invalid}
Let $\Phi_1\in \mathbb{Z}^{n,p}$ and let 
\begin{align} \label{T2iT2j}
    &\Phi_1 \xrightarrow{T^{i}_2} \Phi_2 \xrightarrow{\text{\circled{$T^{k}_2$}}} \Phi_3
\end{align} 
be a sequence where $T_2^i$ is valid  and  $T^{k}_2$ is invalid. If we  swap both moves then  it is not possible to obtain a sequence where both moves become invalid.
\end{lemma}

\begin{proof}
For convenience we will employ the notation $\Phi'_2$ for the state $T^{k}_2(\Phi_1)$. 

As $T^k_2$ is invalid in sequence~\eqref{T2iT2j}, then $i\neq k$. So we have  two possibilities:
\begin{enumerate}[(i)]
    \item The $i^{th}$ position and the $k^{th}$ position belongs to different tuples of $\Phi_1$. 
    
    Then the moves do not interfere with each other. So after the swap we have 
    $$\Phi_1 \xrightarrow{\text{\circled{$T^{k}_2$}}}  \Phi'_2 \xrightarrow{T^{i}_2}   \Phi_3.$$
    \item The $i^{th}$ position and the $k^{th}$ position belongs to the same tuple of $\Phi_1$. 

    As the $i^{th}$ position is dominant in $\Phi_1$, then we have two possibilities:
    \begin{enumerate}
        \item \label{iia} $\Phi_1[i]=\Phi_1[k]$. Then 
    $$\Phi_1 \xrightarrow{T^{k}_2} \Phi'_2 \xrightarrow{\text{\circled{$T^{i}_2$}}} \Phi_3$$
    as the $k^{th}$ position is dominant in $\Phi_1$ and the  $i^{th}$ position is not dominant in $\Phi'_2$.

            \item $\Phi_1[i]>\Phi_1[k]$. Then 
    $$\Phi_1 \xrightarrow{\text{\circled{$T^{k}_2$}}} \Phi'_2 \xrightarrow{T^{i}_2} \Phi_3$$
    as the $k^{th}$ position is not dominant in $\Phi_1$ and the  $i^{th}$ position is  dominant in $\Phi'_2$.
    \end{enumerate}

\end{enumerate}

\end{proof}

\begin{lemma} \label{T3validT2invalid}
Let $\Phi_1\in \mathbb{Z}^{n,p}$  and let 
\begin{align} \label{T2T3}
   &\Phi_1 \xrightarrow{\text{$T^{ij}_3$}} \Phi_2 \xrightarrow{\text{\circled{$T^{k}_2$}}} \Phi_3,
\end{align}
be a sequence where  $T^{ij}_3$ is valid  and $T^{k}_2$ is invalid. If $k\neq j$ and  we  swap both moves then  it is not possible to obtain a sequence where both moves become invalid.
\end{lemma}

\begin{proof}
For convenience we will employ the notation $\Phi'_2$ for the state $T^{k}_2(\Phi_1)$. 

Note that  $i\neq k$, otherwise the second move $T^k_2$
in sequence~\eqref{T2T3} would be valid.

As $T^{ij}_3$ is valid  for $\Phi_1$ then the $i^{th}$ and $j^{th}$ positions on $\Phi_1$ belongs to the same tuple. Now we analyze the cases  depending on  the $i^{th}$ and $k^{th}$ positions of $\Phi_1$:
\begin{enumerate}[(i)]
    \item On $\Phi_1$ the $i^{th}$ and  $k^{th}$ positions are on different tuples. Then $T_3^{ij}$ and $T_2^k$ do not interfere with each other, and after a swap we obtain the sequence
\begin{align*} 
   &\Phi_1 \xrightarrow{\text{\circled{$T^{k}_2$}}} \Phi_2' \xrightarrow{\text{$T^{ij}_3$}} \Phi_3.
\end{align*}
    
    \item On $\Phi_1$ the $i^{th}$ and  $k^{th}$ positions are on the same tuple. As $T^{ij}_3$ is valid  for $\Phi_1$  then  the $i^{th}$ position is  dominant in $\Phi_1$. Consider  two possibilities: 
\begin{enumerate}
\item \label{L8-valid-invalid-case} If $\Phi_1[i]=\Phi_1[k]$  we obtain 
\[\Phi_1 \xrightarrow{\text{$T^{k}_2$}} \Phi_2' \xrightarrow{\text{\circled{$T^{ij}_3$}}} \Phi_3
\]
with $T^{k}_2$ being valid for $\Phi_1$, and where $T^{ij}_3$ is invalid for $\Phi_2'$ since  $\Phi_2'[i]<\Phi_2'[k]$.
\item If $\Phi_1[i]>\Phi_1[k]$  we obtain 
\[\Phi_1 \xrightarrow{\text{\circled{$T^{k}_2$}}} \Phi_2' \xrightarrow{\text{$T^{ij}_3$}} \Phi_3
\]
with $T^{k}_2$ being invalid for $\Phi_1$, and where  $T^{ij}_3$  is  valid for $\Phi_2'$ since $\Phi'_2[i]\geq \Phi'_2[k]$.
\end{enumerate}
\end{enumerate}
\end{proof}

In Lemma~\ref{T3validT2invalid} the condition $k\neq j$  is necessary. For instance, in 
\begin{align*} 
   &(0,1) \xrightarrow{\text{$T^{21}_3$}} (-1,2) \xrightarrow{\text{\circled{$T^{1}_2$}}} (0,2),
\end{align*}
both moves become invalid after a swap 
\begin{align*} 
   &(0,1) \xrightarrow{\text{\circled{$T^{1}_2$}}} (1,1) \xrightarrow{\text{\circled{$T^{21}_{3}$}}} (0,2).
\end{align*}

From Lemmas~\ref{T1validT2invalid},~\ref{T2validT2invalid} and~\ref{T3validT2invalid} it follows the following result that will play an important role in the proof of the main result.

\begin{corollary} \label{MvalidT2invalid}
Let $\Phi_1\in \mathbb{Z}^{n,p}$  and let 
\begin{align*} 
   &\Phi_1 \xrightarrow{\text{$M$}} \Phi_2 \xrightarrow{\text{\circled{$T^{k}_2$}}} \Phi_3
\end{align*}
be a sequence where $T^{k}_2$ is invalid for $\Phi_2$ and where $M$ is a valid move of type 1, 2 or 3  different from $T^{ik}_3$. If   we  swap both moves then  it is not possible to obtain a sequence where both moves become invalid.
\end{corollary}

\section{A technical lemma} \label{Technical_Lemmas_2}

This very technical lemma and its corollaries are presented here, instead of including them as part of the proof of  the main theorem, to streamline that proof.

\begin{lemma} \label{decrease-positive-increase-negative}
Let $\Phi$ and $\Phi'$ be two  states of $\mathbb{Z}^{n,d}$ that coincide (including their partitions into tuples) except at the $a^{th}$  position   where $\Phi[a] > \Phi'[a] \geq 0$ and at the $b^{th}$ position where  $\Phi[b] <  \Phi'[b] \leq  0$. If $$\Phi \xrightarrow{M}  \Psi$$ 
is a valid move  of type 1,  2 or  3 such that $\Phi[a]=\Psi[a]$,  then  
$$\Phi' \xrightarrow{M} \Psi'$$
is a valid move.
\end{lemma}

\begin{proof} For each type of move we will prove that $M$ is valid for $\Phi'$.

  \begin{enumerate}[(a)]
\item Suppose that $M$ is a type 1 move. 

As $\Phi$ and $\Phi'$ have equal number of tuples and $M$ is valid for $\Phi$, then $M$ is  valid for $\Phi'$.

\item Suppose that $M=T_2^i$. 

 As $\Psi[a]=\Phi[a]$ then $i\neq a$, and as $\Phi[b] < 0$ then $i\neq b$.


As $T_2^i$ is valid for $\Phi$ then  the $i^{th}$ position is dominant  in $\Phi$. As   $\Phi'[i]=\Phi[i]$ and for each $k=1,\ldots,n$ we have that $0 \leq |\Phi'[k]|\leq  |\Phi[k]|$,  then  the $i^{th}$ position is also dominant  in $\Phi'$. So   $T_2^i$ is valid for $\Phi'$.

\item Suppose that $M=T_3^{ij}$. 

As $\Psi[a]=\Phi[a]$ then $i\neq a$, and as $\Phi[a] > 0$ then $j\neq a$.

If $T_3^{ij}$ is  valid for $\Phi$, then the $i^{th}$ and $j^{th}$ positions are in the same tuple, and the $i^{th}$ position is dominant  in $\Phi$. That $T_3^{ij}$ is  valid for $\Phi'$ follows from the  claims:
\begin{enumerate}[(i)]

    \item The $i^{th}$ and $j^{th}$ positions of $\Phi$, and therefore of $\Phi'$, are in the same tuple. 
    \item The $i^{th}$ position is dominant in $\Phi'$ (arguing as in (b)).
    \item $\Phi'[j]\leq 0$:  if $j=b$ then $\Phi'[b]\leq 0$, and  if $j\neq b$ then $\Phi'[j]=\Phi[j]\leq 0$.
\end{enumerate}

\end{enumerate}  
\end{proof}

We will write two versions of the previous Lemma where instead of two non-coincident positions in $\Phi$ and $\Phi'$ there is only one non-coincident position. The proofs are the same except that in the proof of Corollary~\ref{Lemma-Phi2PhiPrimaT4} we would omit the reference to the $b^{th}$ position, and in the proof of Corollary~\ref{Lemma-Phi2PhiPrima} we would omit the reference to the $a^{th}$ position.

\begin{corollary} \label{Lemma-Phi2PhiPrimaT4}
Let $\Phi$ and $\Phi'$ two states that  coincide (including their partitions into tuples)  except for the $a^{th}$  position where $\Phi[a]>\Phi'[a]\geq 0$. If $\Phi \xrightarrow{M}  \Psi$ is a valid move  of type 1,  2 or  3 such that $\Psi[a]=\Phi[a]$,   then  $\Phi' \xrightarrow{M} \Psi'$  is  valid.
\end{corollary}

\begin{corollary} \label{Lemma-Phi2PhiPrima}
Let $\Phi$ and $\Phi'$ two states that coincide (including their partitions into tuples)  except for  the $b^{th}$  position where $\Phi[b]< \Phi'[b]\leq 0$. If $\Phi \xrightarrow{M}  \Psi$ is a valid move  of Type 1,  2 or  3,  then  $\Phi' \xrightarrow{M} \Psi'$ is  valid.
\end{corollary}


\section{Main Theorem}{\label{MainTheorem}}

\begin{theorem}{\label{avoidingtype456}}
A multiset of integer numbers is $C_{\mathbb{Z},6}-$realizable if and only if it is $C_{\mathbb{Z},3}-$realizable. 
\end{theorem}

\begin{proof} 
By Lemma~\ref{Crealiz_C6realiz} and Corollary~\ref{Crealiz_C3realiz}, we only need to prove the statement of the theorem for states, namely, a tuple of integer numbers is $C^*_{\mathbb{Z},6}-$realizable if and only if it is $C^*_{\mathbb{Z},3}-$realizable.

Clearly, if a tuple of integer numbers is $C^*_{\mathbb{Z},3}-$realizable, then it is $C^*_{\mathbb{Z},6}-$realizable. 

The rest of the proof will be devoted to prove that whenever a tuple of integer numbers is $C^*_{\mathbb{Z},6}-$realizable, then it is also $C^*_{\mathbb{Z},3}-$realizable. And it is enough to prove this with any sequence 
\begin{align} \label{original}
    &\Phi_1 \xrightarrow{M_1} \Phi_2 \xrightarrow{M_2}  \cdots  \xrightarrow{M_{k-1}}\Phi_{k}  \xrightarrow{M_{k}} \Lambda
\end{align}
where $\Phi_1=\big((0),\ldots,(0)\big)$;   $M_1, \ldots, M_{k-1}$  are valid moves of type 1 to 3; and  $M_k$ is a valid move of type 4 to 6. We divide the proof according to the type of move that $M_k$ is:

\begin{description} 

 \item \framebox{$\mathbf{Type \ 4: \ M_k=T^{ij}_4}$}
  The sequence~(\ref{original})  becomes  
\begin{align} \label{original_T4ij}
    &\Phi_1 \xrightarrow{M_1} \Phi_2\xrightarrow{M_2} \cdots  \xrightarrow{M_{k-1}}\Phi_{k}  \xrightarrow{T^{ij}_4} \Lambda.
\end{align}

      Therefore the $i^{th}$ position is dominant in $\Phi_k$ with  $ \Phi_k[i]\geq \Phi_k[j]> 0$. And as $M_1,\ldots,M_k$ are of type 1 to 3 then, by the increasing-decreasing law, the values on the  $j^{th}$ position of states along the sequence $\Phi_1 \rightarrow  \cdots \rightarrow \Phi_{k}$  are non-decreasing. 
      
      Let $h$, with $1\leq h< k$, be the  unique integer such that 
      \begin{align}\label{Phi_kj=Phi_hj}
       \Phi_{h}[j]+1=\Phi_{h+1}[j]=\cdots=\Phi_{k}[j]>0.    
      \end{align}

     We analyze the only two  possibilities for $M_h$:

\begin{description}
    
\item \framebox{
   $\mathbf{4A: \ M_h=T_2^j}$}
   The  sequence~(\ref{original_T4ij})  becomes 
    \begin{equation} \label{seq_4A}
    \Phi_1 \xrightarrow{M_1} \cdots \xrightarrow{M_{h-1}} \Phi_{h} \xrightarrow{T^j_2} \Phi_{h+1} \xrightarrow{M_{h+1}} \Phi_{h+2} \rightarrow\cdots\rightarrow  \Phi_{k-1} \xrightarrow{M_{k-1}} \Phi_{k} \xrightarrow{T^{ij}_4} \Lambda.  
    \end{equation}

    Deleting $T_2^j$  and replacing  $T^{ij}_4$ by $T^{i}_2$ we obtain the sequence of moves
    \begin{equation} \label{seq_4A'}
        \Phi_1 \xrightarrow{M_1} \cdots \xrightarrow{M_{h-1}} \Phi_{h} =\Phi'_{h+1} \xrightarrow{\text{\circledD{$M_{h+1}$}}} \Phi'_{h+2} \rightarrow\cdots\rightarrow  \Phi'_{k-1}\xrightarrow{\text{\circledD{$M_{k-1}$}}} \Phi'_{k} \xrightarrow{\text{\circledD{$T_2^i$}}} \Lambda
    \end{equation}
    that only contains moves of types 1 to 3.  By convenience with the notation, in \eqref{seq_4A'} we have also named the state $\Phi_h$ by $\Phi'_{h+1}$. The sequences~(\ref{seq_4A}) and~(\ref{seq_4A'}) end in the same state $\Lambda$ since the only position that is affected by the changes is the $j^{th}$ position. In (\ref{seq_4A}) the $j^{th}$ position is increased by one with $T_2^j$ and then decreased by one with $T_4^{ij}$. We cancel out this increase-decrease in~(\ref{seq_4A'}) by deleting $T_2^j$  and replacing  $T^{ij}_4$ by $T^{i}_2$.
    
    That $\Lambda$ is $C^*_{\mathbb{Z},3}-$realizable follows if all the moves of~(\ref{seq_4A'}) are valid:

    \begin{enumerate}

    \item \emph{$M_{1},\ldots,M_{h-1}$ are valid  in~(\ref{seq_4A'}). }

    These moves are equal in~(\ref{seq_4A}) and in~(\ref{seq_4A'}).

\item \emph{$M_{h+1},\ldots,M_{k-1}$ are valid  in~(\ref{seq_4A'}). }

The elimination of $T_2^j$ in sequence \eqref{seq_4A} makes that in sequence \eqref{seq_4A'}  the states $\Phi'_{h+1},\ldots \Phi'_k$  are obtained, respectively, from $\Phi_{h+1},\ldots \Phi_k$ just by subtracting  one on the $j^{th}$ position. So for $t=h+1,\ldots,k-1$ the states  $\Phi_{t}$ and $\Phi'_{t}$ are equal  except that  $\Phi'_{t}[j]=\Phi_{t}[j]-1$. From  \eqref{Phi_kj=Phi_hj} it follows that $\Phi_{t}[j]>\Phi'_{t}[j]\geq 0$. As $\Phi_{t} \xrightarrow{M_{t}} \Phi_{t+1}$  is a valid move of type 1 to 3 with  $\Phi_{t+1}[j]=\Phi_{t}[j]$ then,  by Corollary~\ref{Lemma-Phi2PhiPrimaT4}, $\Phi'_{t} \xrightarrow{M_{t}} \Phi'_{t+1}$ is also valid.
 
\item \emph{$T_2^i$ is  valid  in~(\ref{seq_4A'}). }

    As   $\Phi_k \xrightarrow{T_4^{ij}}  \Lambda$ is  valid  \eqref{seq_4A}, then the $i^{th}$ position is dominant  in $\Phi_k$. Note that the $i^{th}$ position is also dominant in $\Phi'_k$: the $j^{th}$ position (the only difference between $\Phi_k$ and $\Phi'_k$) is not a problem since $\Phi_{k}[j]>\Phi'_{k}[j]\geq 0$. So $\Phi'_k \xrightarrow{T_2^i}  \Lambda$ is  valid. 

    \end{enumerate}

   \item \framebox{$\mathbf{4B : \ M_h=T^{jg}_3}$}
      The  sequence~(\ref{original})  becomes 
    \begin{equation} \label{seq_4B}
    \Phi_1 \xrightarrow{M_1} \cdots \xrightarrow{M_{h-1}} \Phi_{h} \xrightarrow{T^{jg}_3} \Phi_{h+1} \xrightarrow{M_{h+1}} \Phi_{h+2} \rightarrow \cdots \rightarrow \Phi_{k-1} \xrightarrow{M_{k-1}} \Phi_{k} \xrightarrow{T^{ij}_4} \Lambda.
    \end{equation}
Let us first see that $g\neq i$. Note that  $\Phi_{h+1}[g]<0$ and that, by the increasing-decreasing law,  $\Phi_k[g]<0$. Then $g=i$ would imply that $T_4^{ij}$  is invalid for $\Phi_k$. A contradiction.

Deleting  $T_3^{jg}$  and replacing  $T^{ij}_4$ by $T^{ig}_3$ we obtain
    \begin{equation} \label{seq_4B'}
        \Phi_1 \xrightarrow{M_1} \cdots \xrightarrow{M_{h-1}} \Phi_{h} =\Phi'_{h+1}\xrightarrow{\text{\circledD{$M_{h+1}$}}}\Phi'_{h+2} \rightarrow \cdots  \rightarrow\Phi'_{k-1}\xrightarrow{\text{\circledD{$M_{k-1}$}}} \Phi'_{k} \xrightarrow{\text{\circledD{$T^{ig}_3$}}} \Lambda
    \end{equation}
    that only contains moves of types 1 to 3. By convenience with the notation, in \eqref{seq_4B'} we have also named the state $\Phi_h$ by $\Phi'_{h+1}$. The sequences~(\ref{seq_4B}) and~(\ref{seq_4B'}) end in the same state $\Lambda$ since the only positions that are affected by the changes are the $j^{th}$ and $g^{th}$ positions. In (\ref{seq_4B}) the $j^{th}$ position is increased by one with $M_h=T_3^{jg}$ and then decreased by one with $M_k=T_4^{ij}$, and  we cancel out this increase-decrease in~(\ref{seq_4B'}) by deleting $T_3^{jg}$  and replacing  $T^{ij}_4$ by $T^{ig}_3$. In (\ref{seq_4B}) the $g^{th}$ position is decreased by one with $M_h=T_3^{jg}$, and although we delete $T_3^{jg}$ in (\ref{seq_4B'}) we obtain this decrease  at the end of the sequence with $T_3^{ig}$.

Note that $\Phi'_{h+1}[j]=\Phi_{h}[j]=\Phi_{h+1}[j]-1 \geq 0$, then from  \eqref{Phi_kj=Phi_hj} it follows that 
     \begin{align}\label{j-increase-prime-1}
     &\ \Phi_{h+1}[j]= \Phi_{h+2}[j]=\cdots= \Phi_{k-1}[j] =\Phi_{k}[j]\nonumber\\
      & \quad \ \ \vee \\
     &\ \Phi'_{h+1}[j]= \Phi'_{h+2}[j] = \cdots=\Phi'_{k-1}[j] =\Phi'_{k}[j] \geq 0.\nonumber
     \end{align}
Also note that  $\Phi'_{h+1}[g]=\Phi_h[g]=\Phi_{h+1}[g]+1 \leq 0$ and  that  $\Phi'_{t}[g]=\Phi_{t}[g]+1$ for $t=h+2,\ldots,k$. Therefore
     \begin{align} \label{k-decrease-prime-2}
     & 0 >  \ \Phi_{h+1}[g] \geq \Phi_{h+2}[g]\geq\cdots\geq  \Phi_{k-1}[g] \geq \Phi_{k}[g] \nonumber \\
      & \qquad \quad   \wedge \qquad \qquad  \wedge \qquad \;  \cdots \qquad \; \wedge \qquad \quad \wedge \\
     &0 \geq \ \Phi'_{h+1}[g]\geq  \Phi'_{h+2}[g] \geq  \cdots\geq \Phi'_{k-1}[g] \geq \Phi'_{k}[g]. \nonumber
     \end{align}

    That $\Lambda$ is $C^*_{\mathbb{Z},3}-$realizable follows if all the moves of~(\ref{seq_4B'}) are valid:

    \begin{enumerate}

    \item \emph{$M_{1},\ldots,M_{h-1}$  are valid in~(\ref{seq_4B'}). }

    These moves are equal in~(\ref{seq_4B}) and in~(\ref{seq_4B'}).

 \item \emph{$M_{h+1},\ldots,M_{k-1}$ are valid  in~(\ref{seq_4B'}). }

For $t=h+1,\ldots,k-1$ the states  $\Phi_{t}$ and $\Phi'_{t}$ are equal  except that   $\Phi_{t}[j]>\Phi'_{t}[j]\geq 0$ (see  \eqref{j-increase-prime-1}) and $\Phi_{t}[g] < \Phi'_{t}[g]\leq 0$ (see   \eqref{k-decrease-prime-2}). As $\Phi_{t} \xrightarrow{M_{t}} \Phi_{t+1}$  is  a valid move of type 1 to 3 with  $\Phi_{t+1}[j]=\Phi_{t}[j]$ then,  by Lemma~\ref{decrease-positive-increase-negative}, $\Phi'_{t} \xrightarrow{M_{t}} \Phi'_{t+1}$ is also valid.

    \item \emph{$T_3^{ig}$ is valid  in~(\ref{seq_4B'}). }

    First note that the states  $\Phi_{k}$ and $\Phi'_{k}$ are equal  except that   $\Phi_{k}[j]>\Phi'_{k}[j]\geq 0$ (see \eqref{j-increase-prime-1}) and $\Phi_{k}[g] < \Phi'_{k}[g]\leq 0$ (see \eqref{k-decrease-prime-2}. Now we verify the three conditions for $T_3^{ig}$ to be valid:
\begin{enumerate}[(a)]
    \item The $i^{th}$ position is dominant in $\Phi'_k$.
    
    As   $\Phi_k \xrightarrow{T_4^{ij}}  \Lambda$ is valid  in \eqref{seq_4B}, then the $i^{th}$ position is dominant  in $\Phi_k$. Note that the $i^{th}$ position, which verifies $j\neq i\neq g$, is also dominant in $\Phi'_k$: 
the $j^{th}$ and $g^{th}$ positions (the only difference between $\Phi_k$ and $\Phi'_k$) are not a problem since $|\Phi_k[j]|>|\Phi'_k[j]| \geq  0$ and   $|\Phi_k[g]|>|\Phi'_k[g]|\geq  0$.
  
    \item The $i^{th}$ and $g^{th}$ positions of state $\Phi'_k$ are in the same tuple:  
    \begin{itemize}
    \item The $i^{th}$ and $j^{th}$ positions   are in the same tuple of $\Phi_k$ since  $T_4^{ij}$ is  valid in sequence \eqref{seq_4B}. 
    
    \item The $j^{th}$ and $g^{th}$ positions  are in the same tuple of  $\Phi_h$ since  $T_3^{jg}$ is valid in sequence \eqref{seq_4B}. And, by the permanence law, they are also  in the same tuple of $\Phi_k$.
    
    \item By transitivity the $i^{th}$ and $g^{th}$ positions are in the same tuple of $\Phi_k$.

\item The only moves that alter the structure of the tuples in a state are type 1 moves. So,  $\Phi'_{k}$ and $\Phi_{k}$ have the same structure of tuples since the only difference between both states comes from deleting the move  $T_3^{jg}$ in sequence \eqref{seq_4B}. So the $i^{th}$ and $g^{th}$ positions are in the same tuple of $\Phi'_k$.

    \end{itemize}

    \item $\Phi'_k[g]\leq 0$, as can be seen in \eqref{k-decrease-prime-2}.
\end{enumerate}    

\end{enumerate}

\end{description}

    \item \framebox{$\mathbf{Type \ 5: \ M_k=T_5^{ij}}$}
        The  sequence~(\ref{original})  becomes 
\begin{align} \label{original_T5ij}
    &\Phi_1 \xrightarrow{M_1} \Phi_2 \xrightarrow{M_2}  \cdots  \xrightarrow{M_{k-1}}\Phi_{k}  \xrightarrow{T^{ij}_5} \Lambda.
\end{align}
    Therefore the $i^{th}$ position is dominant in $\Phi_k$, with $\Phi_k[j]< 0$ and $\Phi_k[i]\geq |\Phi_k[j]|> 0$.

As $M_1,\ldots,M_{k-1}$ are of type 1 to 3 then, by the increasing-decreasing law, the values on the  $j^{th}$ position of states along the sequence $\Phi_1 \rightarrow  \cdots \rightarrow \Phi_{k}$ are non-increasing. Let $h$, with $1\leq h< k$, be the  unique integer such that 
      \begin{align*}\label{}
       \Phi_{h}[j]-1=\Phi_{h+1}[j]=\cdots=\Phi_{k}[j]<0.    
      \end{align*}
 As $M_h$ is of type 1 to 3, then the only possibility is that $M_h=T_3^{gj}$ for some $g$. The  sequence~(\ref{original_T5ij})  becomes 
    \begin{equation} \label{seq_5A}
    \Phi_1 \xrightarrow{M_1} \cdots \xrightarrow{M_{h-1}} \Phi_{h} \xrightarrow{T^{gj}_3} \Phi_{h+1} \xrightarrow{M_{h+1}} \Phi_{h+2} \rightarrow\cdots\rightarrow  \Phi_{k-1} \xrightarrow{M_{k-1}} \Phi_{k} \xrightarrow{T^{ij}_5} \Lambda.  
    \end{equation}
Replacing $T^{gj}_3$ by $T_2^g$  and   $T^{ij}_5$ by $T^{i}_2$ we obtain the  sequence  
    \begin{equation} \label{seq_5A'}
    \Phi_1 \xrightarrow{M_1} \cdots \xrightarrow{M_{h-1}} \Phi_{h} \xrightarrow{\text{\circledD{$T_2^g$}}} \Phi'_{h+1} \xrightarrow{\text{\circledD{$M_{h+1}$}}} \Phi'_{h+2} \rightarrow\cdots\rightarrow  \Phi'_{k-1} \xrightarrow{\text{\circledD{$M_{k-1}$}}} \Phi'_{k} \xrightarrow{\text{\circledD{$T_2^i$}}} \Lambda
    \end{equation}
that only contains moves of types 1 to 3.     The sequences~(\ref{seq_5A}) and~(\ref{seq_5A'}) end in the same state $\Lambda$ since the only position that is affected by the changes is the $j^{th}$ position. In (\ref{seq_5A}) the $j^{th}$ position is decreased by one with $T_3^{gj}$ and then increased by one with $T_5^{ij}$. We cancel out this decrease-increase in~(\ref{seq_5A'}) by replacing $T^{gj}_3$ by $T_2^g$  and   $T^{ij}_5$ by $T^{i}_2$.   

Note that      
     \begin{align}\label{j-increaseT5-prime}
     \Phi_{h}[j]-1 = &\ \Phi_{h+1}[j]= \Phi_{h+2}[j]=\cdots= \Phi_{k-1}[j] =\Phi_{k}[j] \nonumber \\
    &\quad \wedge  \\
     \Phi_h[j] =&\ \Phi'_{h+1}[j]= \Phi'_{h+2}[j]=\cdots= \Phi'_{k-1}[j] =\Phi'_{k}[j] \leq 0. \nonumber
     \end{align}

That $\Lambda$ is $C^*_{\mathbb{Z},3}-$realizable follows if all the moves of~(\ref{seq_5A'}) are valid:

\begin{enumerate}

\item \emph{$M_{1},\ldots,M_{h-1}$  are valid  in~(\ref{seq_5A'}). }

These moves are equal in~(\ref{seq_5A}) and in~(\ref{seq_5A'}).

\item \emph{$T_2^g$ is valid in~(\ref{seq_5A'}).}

This is so because $T_3^{gj}$ is valid in~(\ref{seq_5A}), and so the $g^{th}$ position is  dominant in $\Phi_h$.

 \item \emph{$M_{h+1},\ldots,M_{k-1}$ are valid  in~(\ref{seq_5A'}). }

For $t=h+1,\ldots,k-1$ the states  $\Phi_{t}$ and $\Phi'_{t}$ are equal  except that   $\Phi_{t}[j]<\Phi'_{t}[j]\leq 0$  (see \eqref{j-increaseT5-prime}). As $\Phi_{t} \xrightarrow{M_{t}} \Phi_{t+1}$  is  a valid move of type 1 to 3  then,  by Corollary~\ref{Lemma-Phi2PhiPrima}, $\Phi'_{t} \xrightarrow{M_{t}} \Phi'_{t+1}$ is also valid.

    \item \emph{$T_2^i$ is valid  in~(\ref{seq_5A'}). }

    As   $\Phi_k \xrightarrow{T_5^{ij}}  \Lambda$ is valid in \eqref{seq_5A}, then the $i^{th}$ position is dominant  in $\Phi_k$. Note that the $i^{th}$ position is also dominant in $\Phi'_k$: the $j^{th}$ position (the only difference between $\Phi_k$ and $\Phi'_k$) is not a problem since $\Phi_k [j]<\Phi'_k[j]\leq 0$ and so $|\Phi'_k [j]|<|\Phi_k[j]|$. Thus $\Phi'_k \xrightarrow{T_2^i}  \Lambda$ is  valid.
    \end{enumerate}


\item \framebox{$\mathbf{Type \ 6: \ M_k=T_6^{ij}}$}   
The original sequence~(\ref{original})  becomes 
    \begin{equation} \label{T6}
    \Phi_1 \xrightarrow{M_1} \Phi_2 \xrightarrow{M_2}  \cdots   \xrightarrow{M_{k-1}} \Phi_{k} \xrightarrow{T^{ij}_6} \Lambda,
    \end{equation}
where $M_1,\ldots,M_{k-1}$  are of type 1 to 3, and different from $T_3^{gj}$ by the increasing-decreasing law.

    Note that $\Phi_k[i]\geq\Phi_k[j]\geq 0$, so we will consider two subcases:
\begin{description}

    \item \label{Phi_k[i]=Phi_k[j]} \framebox{$ \mathbf{6A: \ \Phi_k[i]=\Phi_k[j]\geq 0}$} All the following statements are true:

\begin{enumerate}[(i)] 

    \item     \emph{The  $i^{th}$ and $j^{th}$  positions are   dominant in $\Phi_k$.}
    
    Since $\Phi_k[i]=\Phi_k[j]$  and  $T_6^{ij}$ is valid.

    \item   \emph{The $i^{th}$ and $j^{th}$  positions are  dominant in  $\Phi_{1},\ldots, \Phi_{k-1}$.}
    
    By the dominance law. 

    \item  \emph{There is a move $M_{p}=T_1^{(h)^{\frown}(h+1)}$ in \eqref{T6} such that the $i^{th}$ and $j^{th}$ positions where in different tuples in $\Phi_p$, and $M_p$ puts them in the same tuple.}
    
    The $i^{th}$ and $j^{th}$ positions where in different tuples in $\Phi_1=((0),\ldots, (0))$, and they are in the same tuple in $\Phi_k$, so the claim must be true.  
    Note that \eqref{T6} becomes:
    \begin{equation} \label{T6a}
    \Phi_1 \xrightarrow{M_1} \cdots  \xrightarrow{M_{p-1}} \Phi_p  \xrightarrow{T_1^{(h)^{\frown}(h+1)}} \Phi_{p+1} \xrightarrow{M_{p+1}}  \cdots \xrightarrow{M_{k-1}} \Phi_{k} \xrightarrow{T^{ij}_6} \Lambda.
    \end{equation}

    \item  \emph{$\Phi_{p+1}[i]=\Phi_{p+1}[j]$, \ldots, $\Phi_{k}[i]=\Phi_{k}[j]$.}
    
    Since the $i^{th}$ and $j^{th}$ positions are dominant in $\Phi_{p+1}, \ldots, \Phi_k$ by (i) and (ii), and belong to the same tuple in all those states by (iii).

    \item \emph{The values at the $i^{th}$ and $j^{th}$ positions in  $\Phi_{p+1}, \ldots, \Phi_k$  remain unchanged.}
    
    Otherwise, by (iv), both values increase simultaneously  through moves of type 6. This is not possible  since  $M_{p+1},\ldots,M_{k-1}$ are moves of type 1 to 3. 
    
    \item  \emph{$M_{p+1}$ $, \ldots, M_{k-1}$ can not be moves of type 2 or 3 that involve the tuple to which the $i^{th}$ and $j^{th}$ positions belong.}
    
    By (i), (ii), and  (v),  the values in the tuple to which the $i^{th}$ and $j^{th}$ positions belong in states $\Phi_{p+1}, \ldots, \Phi_k$ remain unchanged. 

\end{enumerate}

Now we  modify sequence \eqref{T6a} to the sequence
    \begin{equation} \label{T6a2}
    \Phi_1 \xrightarrow{M_1} \cdots  \xrightarrow{M_{p-1}} \Phi_p \xrightarrow{\text{\circledD{$T_2^i$}}}  \xrightarrow{\text{\circledD{$T_2^j$}}} \Phi'_{p} \xrightarrow{T_1^{(h)^{\frown}(h+1)}}
     \Phi'_{p+1} \xrightarrow{\text{\circledD{$M_{p+1}$}}}   \cdots \xrightarrow{\text{\circledD{$M_{k-1}$}}} \Phi'_k=\Lambda.
    \end{equation}
The sequences~(\ref{T6a}) and~(\ref{T6a2}) end in the same state $\Lambda$ since the only positions that are affected by the changes are the $i^{th}$ and $j^{th}$ positions. In (\ref{T6a}) the $i^{th}$ and $j^{th}$ positions are increased by one with $T_6^{ij}$, and in \eqref{T6a2} the increase of the deleted $T_6^{ij}$ is realized by  $T_2^{i}$ and $T_2^{j}$. 

That $\Lambda$ is $C^*_{\mathbb{Z},3}-$realizable follows if all the moves of~(\ref{T6a2}) are valid: 

\begin{enumerate}
 \item   \emph{$M_1, \ldots, M_{p-1}$ are valid  in \eqref{T6a2}.} 

These moves are equal in~(\ref{T6a}) and in~(\ref{T6a2}).

\item \emph{$T_2^i$ and $T_2^j$ are  valid in \eqref{T6a2}.}  

This is so because the $i^{th}$ and $j^{th}$ positions are dominant in $\Phi_p$ and belong to different tuples  (see (ii) and (iii)).

\item \emph{$M_{p+1}, \ldots, M_{k-1}$ are valid  in \eqref{T6a2}.}

The only differences of $\Phi_{p+1}$ and $\Phi'_{p+1}$ are on the $i^{th}$ and $j^{th}$ positions with $\Phi'_{p+1}[i]=\Phi_{p+1}[i]+1$ and $\Phi'_{p+1}[j]=\Phi_{p+1}[j]+1$. So, by item (vi),  the moves  $M_{p+1}, \ldots, M_{k-1}$  are valid.
\end{enumerate}

\item \framebox{$\mathbf{6B. \ \Phi_k[i]>\Phi_k[j]\geq 0}$} \label{Phi_k[i]>Phi_k[j]}

Taking into account that $T_6^{ij}(\Phi_k)=T_2^i(T_2^j(\Phi_k))$ and that $\Phi_k[i]>\Phi_k[j]\geq 0$,  if in sequence~(\ref{T6}) we replace  $T^{ij}_6$ by two consecutive moves $T_2^j$ and $T_2^i$  we  obtain 
\begin{equation} \label{T6->T2T2}
   \Phi_1 \xrightarrow{M_1}   \cdots  \xrightarrow{M_{k-2}} \Phi_{k-1} \xrightarrow{M_{k-1}} \Phi_{k} \xrightarrow{\text{\circled{$T^{j}_2$}}} T_2^j(\Phi_k) \xrightarrow{T^{i}_2} \Lambda
\end{equation}
where $M_1,\ldots,M_{k-1}$  are of type 1 to 3 and different from $T_3^{gj}$. Note that $T^{j}_2$ is invalid for $\Phi_k$, and so $T_2^j(\Phi_k)$ might not be $C^*_{\mathbb{Z},3}-$realizable. So $T_2^i$ is a valid move applied to a non-realizable state, but this is not a contradiction since the validity of a move does not depend on the realizability of the state in which applies.

Now we start an algorithmic procedure, where the first step is to swap $M_{k-1}$ and $T_2^j$. According to Corollary~\ref{MvalidT2invalid}, we obtain one of the following possibilities:
\begin{enumerate}[(a)]
\item The sequence of valid moves
\begin{equation} \label{Free_Of_Invalid}
   \Phi_1 \xrightarrow{M_1}   \cdots  \xrightarrow{M_{k-2}} \Phi_{k-1} \xrightarrow{T^{j}_2} \Phi'_{k} \xrightarrow{M_{k-1}}  \Phi'_{k+1} \xrightarrow{T^{i}_2} \Lambda.
\end{equation}
If this was the case, the procedure stops. 

\item  The sequence of  moves
    \begin{equation} \label{Swap_The_Invalid}
   \Phi_1 \xrightarrow{M_1}   \cdots  \xrightarrow{M_{k-2}} \Phi_{k-1} \xrightarrow{T^{j}_2} \Phi'_k \xrightarrow{\text{\circled{$M_{k-1}$}}}  \Phi'_{k+1} \xrightarrow{T^{i}_2} \Lambda.
\end{equation}
  If this is the case, then the procedure also stops.

\item  \label{itemCofProcedure} The sequence of  moves
    \begin{equation} \label{Avance_The_Invalid}
   \Phi_1 \xrightarrow{M_1}   \cdots  \xrightarrow{M_{k-2}} \Phi_{k-1} \xrightarrow{\text{\circled{$T^{j}_2$}}} \Phi'_{k} \xrightarrow{M_{k-1}}  \Phi'_{k+1} \xrightarrow{T^{i}_2} \Lambda,
\end{equation}
where the situation is similar to the one in sequence~\eqref{T6->T2T2} with the invalid move $T_2^j$ in a previous position in the sequence of moves.  Then the procedure continues recursively by swapping $M_{k-2}$ and $T_2^j$ in \eqref{Avance_The_Invalid} obtaining again three possibilities: a sequence similar to \eqref{Free_Of_Invalid}, to \eqref{Swap_The_Invalid} or to \eqref{Avance_The_Invalid}.  We repeat this argument recursively each time that we achieve a situation similar to  the one in sequence~\eqref{Avance_The_Invalid}. 

\end{enumerate}

 Note that the procedure does not finish in a situation similar to~\eqref{itemCofProcedure} since  if we take $T_2^j$ all the way back to the beginning of the sequence it becomes valid since $T_2^j$ is valid for $\Phi_1=\big((0),\ldots,(0)\big)$. So the procedure would finish in one of the following situations:
\begin{enumerate}[(a')]
    \item A sequence of valid moves:
\begin{equation} \label{Free_Of_Invalid_Finishitem }
   \Phi_1 \xrightarrow{M_1}  \cdots\xrightarrow{M_{p-1}} \Phi_p \xrightarrow{T^{j}_2}\xrightarrow{M_p}\xrightarrow{M_{p+1}}  \cdots\xrightarrow{M_{k-1}}   \xrightarrow{T^{i}_2} \Lambda
\end{equation}
If this was the case, then we have reached the main goal:  the state $\Lambda$ is $C^*_{\mathbb{Z},3}-$realizable since it  is obtained starting with $\Phi_1=\big((0),\ldots,(0)\big)$ by means of a sequence valid moves of type 1 to 3.

\item A sequence 
\begin{equation} \label{after_last_swap}
   \Phi_1 \xrightarrow{M_1}  \cdots\xrightarrow{M_{p-1}} \Phi_p \xrightarrow{T^{j}_2}\xrightarrow{\text{\circled{$M_p$}}}\xrightarrow{M_{p+1}}  \cdots\xrightarrow{M_{k-1}}   \xrightarrow{T^{i}_2} \Lambda.
\end{equation}
where all moves are of type 1 to 3.

\end{enumerate}

Now, in the next pages, we will analyse thoroughly   what to do when we arrive to sequence~(\ref{after_last_swap}) depending on what type of move is $M_p$. Note that  $M_p$ is not of type 1 due to  Lemma~\ref{T1validT2invalid} and so we will analyze the other two possibilities:


\begin{description}

    \item  $\mathbf{6B.I.}$ \ $M_p=T_2^g$.  
    
    Then sequence~(\ref{after_last_swap}) is
\begin{equation*} 
   \Phi_1 \xrightarrow{M_1}  \cdots\xrightarrow{M_{p-1}} \Phi_p \xrightarrow{T^{j}_2}\xrightarrow{\text{\circled{$T_2^g$}}}\xrightarrow{M_{p+1}}  \cdots\xrightarrow{M_{k-1}}   \xrightarrow{T^{i}_2} \Lambda. 
\end{equation*}
Note that   the last of the swaps of $T_2^j$ is as follows:
$$
\Phi_p \xrightarrow{T_2^{g}} \xrightarrow{\text{\circled{$T^{j}_2$}}}\qquad \Rightarrow \qquad  \Phi_p \xrightarrow{T^{j}_2} \xrightarrow{\text{\circled{$T_2^{g}$}}}.
$$
This happens, according to  item \eqref{iia} of the proof of Lemma~\ref{T2validT2invalid},   when the $j^{th}$ and  $g^{th}$ positions on $\Phi_p$ belong to the same tuple and  $\Phi_p[j]=\Phi_p[g]$. Then   we can merge $T_2^j$ and $T_2^g$  into $T_6^{jg}$ to obtain a new sequence:
\begin{equation*} 
   \Phi_1 \xrightarrow{M_1}  \cdots\xrightarrow{M_{p-1}} \Phi_p \xrightarrow{T^{jg}_6}\xrightarrow{M_{p+1}}  \cdots\xrightarrow{M_{k-1}}   \xrightarrow{T^{i}_2} \Lambda. 
\end{equation*}
 where $T_6^{jg}$ is a valid move. And since $\Phi_p[j]=\Phi_p[g]$ then we can proceed as in the subcase 6A
 to prove that $T_6^{jg}(\Phi_p)$ is $C^*_{\mathbb{Z},3}-$realizable, and so $\Lambda$ is also $C^*_{\mathbb{Z},3}-$realizable.
 


\item $\mathbf{6B.II.}$ \ $M_p=T_3^{gh}$.  

The original sequence \eqref{T6}  is then 
\begin{align*} 
  \Phi_1 \xrightarrow{M_1} \cdots \xrightarrow{M_{p-1}}  
  \Phi_p \xrightarrow{T_3^{gh}} \Phi_{p+1} \xrightarrow{M_{p+1}}\cdots \xrightarrow{M_{k-1}}
  \Phi_k \xrightarrow{T^{ij}_6}   \Lambda.
 \end{align*}
Note that $h\neq j$ by the increasing-decreasing law. As the first part of the sequence 
 \begin{align*} 
  \Phi_1 \xrightarrow{M_1} \cdots \xrightarrow{M_{p-1}}  \Phi_p 
 \end{align*}
composed of valid moves of type 1 to 3 will not change in our arguments below, we consider only part of the sequence, namely from  $\Phi_p$ ahead. That is, 
\begin{align} \label{original2}
  \Phi_p \xrightarrow{T_3^{gh}} \Phi_{p+1} \xrightarrow{M_{p+1}}\cdots \xrightarrow{M_{k-1}}
  \Phi_k \xrightarrow{T^{ij}_6}   \Lambda
 \end{align}
that after the swaps of $T_2^j$  becomes 
\begin{align} \label{invalidWithoutGap}
   \Phi_p \xrightarrow{T^{j}_2} \Psi_{p+1}\xrightarrow{\text{\circled{$T_3^{gh}$}}}\Psi_{p+2}\xrightarrow{M_{p+1}} \cdots \xrightarrow{M_{k-1}}   \Psi_{k+1}\xrightarrow{T^{i}_2} \Lambda.
\end{align}
 Note that in  the last of the swaps of $T_2^j$ we have
$$
\Phi_p \xrightarrow{T_3^{gh}} \Phi_{p+1}\xrightarrow{\text{\circled{$T^{j}_2$}}}\Psi_{p+2}\quad \Rightarrow \quad  \Phi_p \xrightarrow{T^{j}_2} \Psi_{p+1}\xrightarrow{\text{\circled{$T_3^{gh}$}}}\Psi_{p+2}
$$
 that corresponds to item \eqref{L8-valid-invalid-case} of the proof of Lemma~\ref{T3validT2invalid}. And  so $j \neq g$, the $j^{th}$ and   $g^{th}$ positions in $\Phi_p$ belong to the same tuple, and  $\Phi_p[j]=\Phi_p[g]$.

Consider the modification of~(\ref{invalidWithoutGap}) given by 
\begin{equation} \label{invalidWithoutGap2}
   \Phi_p 
   \xrightarrow{T^{j}_2} \Psi_{p+1}\xrightarrow{\text{\circledD{$T_2^{g}$}}}  \Psi'_{p+2} \xrightarrow{\text{\circledD{$M_{p+1}$}}} \Psi'_{p+3}\xrightarrow{}  \cdots \xrightarrow{\text{\circledD{$M_{k-1}$}}}   \Psi'_{k+1} \xrightarrow{\text{\circledD{$T^{ih}_3$}}}  \Lambda, 
\end{equation}

where $T_3^{gh}$ and $T_2^{i}$ in~\eqref{invalidWithoutGap} are  replaced, respectively, by $T_2^{g}$ and $T_3^{ih}$ in~\eqref{invalidWithoutGap2}. The only position  affected by the changes is the $h^{th}$  position and, clearly, the final state in both sequences is $\Lambda$. 
 
In what follows, we will show that sequence~\eqref{invalidWithoutGap2} is composed of valid moves of type 1 to 3 except the move $T_2^{g}$ that is invalid.  For this we will use the information of the following diagram: 
\begin{align} \label{new-h}
     0>  & \Psi_{p+2}[h]\geq \Psi_{p+3}[h]\geq \cdots\geq  \Psi_{k}[h] \geq\Psi_{k+1}[h] \nonumber \\
      & \quad\;    \wedge \qquad \quad\;  \wedge \qquad \;\;  \cdots \qquad \; \wedge \quad \qquad \wedge \\
    0 \geq & \Psi'_{p+2}[h]\geq \Psi'_{p+3}[h]\geq\cdots\geq \Psi'_{k}[h] \geq\Psi'_{k+1}[h]  \nonumber
     \end{align}     
that we  obtain from the three following items:
\begin{enumerate}[(i)]
\item Let us see first that $0>\Psi_{p+2}[h]$. Note that   
\begin{equation*}\label{Psi_p+2>0}
\Psi_{p+2}[h]=\Psi_{p+1}[h]-1=\Phi_{p}[h]-1<0
\end{equation*}
where the  first equality follows from~\eqref{invalidWithoutGap}, the second equality follows from~\eqref{invalidWithoutGap} since $h\neq j$, and the third inequality follows from~\eqref{original2} since $T_3^{gh}$ is valid for $\Phi_p$ which implies that $\Phi_{p}[h]\leq 0$.
\item After that, as $M_{p+1},\ldots, M_{k-1}$ are valid moves in~\eqref{invalidWithoutGap}, then by the increasing-decreasing law we have that  
\begin{equation*}\label{Psi's}
\Psi_{p+2}[h]\geq \Psi_{p+3}[h]\geq \cdots\geq \Psi_{k+1}[h].
\end{equation*}
    \item And to finish, as $T_3^{gh}$ in~\eqref{invalidWithoutGap} becomes $T_2^g$ in~\eqref{invalidWithoutGap2} then 
\begin{equation*}\label{Psi'=Psi+1}
\Psi'_{t}[h]=\Psi_{t}[h]+1\quad \text{ for }t=p+2,\ldots,k+1.
\end{equation*}

\end{enumerate}

Now we have the  tools for determining  the validity of the moves in sequence~\eqref{invalidWithoutGap2}:

\begin{enumerate}[1.]
\item \emph{$T_2^{j}$ is valid in ~\eqref{invalidWithoutGap2}.}

This is so by hypothesis.

\item \label{T2g-invalid} \emph{$T_2^{g}$ is invalid   in ~\eqref{invalidWithoutGap2}.}

This is so because the $g^{th}$ position is not dominant in $\Psi_{p+1}$. And this follows from the following facts:   $j\neq g$,    the $j^{th}$ and   $g^{th}$ positions in $\Phi_p$ belong to the same tuple, and  $\Phi_p[j]=\Phi_p[g]$. These claims were proved in the paragraph just after sequence~\eqref{invalidWithoutGap}.

\item \emph{$M_{p+1},\ldots,M_{k-1}$ are valid in ~\eqref{invalidWithoutGap2}.}

From diagram~\eqref{new-h} it follows that for $t=p+2,\ldots,k$ the states  $\Psi_{t}$ and $\Psi'_{t}$ are equal  except that   $\Psi_{t}[h]<\Psi'_{t}[h]\leq 0$.   On the other hand, sequence~\eqref{invalidWithoutGap}  contains the move $\Psi_{t} \xrightarrow{M_{t-1}} \Psi_{t+1}$ which is valid  of type 1 to 3. Then,  by Corollary~\ref{Lemma-Phi2PhiPrima}, $\Psi'_{t} \xrightarrow{M_{t-1}} \Psi'_{t+1}$ is also valid.

\item \emph{$T_3^{ih}$ is valid in~\eqref{invalidWithoutGap2}.}

Now we verify the three conditions for $T_3^{ih}$ to be valid:
\begin{itemize}
    \item The $i^{th}$ and $h^{th}$ positions of state $\Psi'_{k+1}$  are in the same tuple:  
    \begin{itemize}
    \item The last move in  sequence  \eqref{original2} is a valid $T_6^{ij}$. This implies that the $i^{th}$ and $j^{th}$ position of states $\Phi_k$ and $\Lambda$ are in the same tuple.
    
    \item As we showed in the paragraph after sequence~\eqref{invalidWithoutGap}, the $j^{th}$ and $g^{th}$ position of $\Phi_p$ are in the same tuple. By the permanence law they also are in the same tuple of $\Lambda$.
    
    \item From sequence~\eqref{original2} it follows that the $g^{th}$ and $h^{th}$ positions of $\Phi_p$ are in the same tuple. By the permanence law they are in the same tuple of $\Lambda$.    
    
    \item By transitivity 
    \begin{align} \label{transitivity-i-h}
    \text{the $i^{th}$ and $h^{th}$ positions of $\Lambda$ are in the same tuple.}
    \end{align}

    \item Independently of its validity, the move $T_3^{ih}$ in \eqref{invalidWithoutGap2} maintains the structure of tuples in $\Psi'_{k+1}$ and $\Lambda$. So the $i^{th}$ and $h^{th}$ positions are also in the same tuple of $\Psi'_{k+1}$.
    \end{itemize}
    
    \item The $i^{th}$ position is dominant in $\Psi'_{k+1}$.

    Comparing sequences~\eqref{original2} and~\eqref{invalidWithoutGap2}, as $j\neq h$,  we have that 
    \begin{align} \label{Phi_k&Delta}
        \begin{matrix}
            \Psi'_{k+1}[j]=\Phi_k[j]+1 \\
            \Psi'_{k+1}[h]=\Phi_k[h]+1 \\
            \Psi'_{k+1}[t]=\Phi_k[t] & \text{for each $t\neq j,h$.}
        \end{matrix}
    \end{align} 
In particular, $\Psi'_{k+1}[i]=\Phi_k[i]$ since $i\neq j,h$. 

The $i^{th}$ position is dominant in $\Phi_k$, see~\eqref{original2}, so the $j^{th}$ and $h^{th}$  are the only positions that could  change the dominance of the $i^{th}$ position in $\Psi'_{k+1}$. Let us see that this is not the case: 
    
    \begin{itemize}

        \item  $\Psi'_{k+1}[i]\geq |\Psi'_{k+1}[j]|$.

We will prove that 
\begin{align*} 
\Psi'_{k+1}[i]=\Phi_k[i]\geq  |\Psi'_{k+1}[j]|.
\end{align*}
The first equality  follows from~\eqref{Phi_k&Delta},  for  the second inequality apply~\eqref{Phi_k&Delta} to  $\Phi_k[i]>\Phi_k[j]\geq 0$ which is  the  hypothesis of this case:~6B. 

    \item  $\Psi'_{k+1}[i]> |\Psi'_{k+1}[h]|$.

We will prove that 
\begin{align*} 
\Psi'_{k+1}[i]=\Phi_k[i]\geq |\Phi_k[h]| > |\Psi'_{k+1}[h]|.
\end{align*}
The first equality  follows from~\eqref{Phi_k&Delta}. 

Let us prove the   second inequality. Note that  the $i^{th}$ and $h^{th}$ positions belong to the same tuple in $\Lambda$, as we have seen in~\eqref{transitivity-i-h}. Since $\Phi_k \xrightarrow{T_6^{ij}}\Lambda$  maintains the structure of the tuples in $\Phi_k$ and $\Lambda$, then  the $i^{th}$ and $h^{th}$ positions also belong to same tuple in $\Phi_k$. And so the   second inequality follows from the dominance of  the $i^{th}$ position in $\Phi_k$.

Finally, let us see the third inequality. Note that $\Phi_p \xrightarrow{T_3^{gh}}\Phi_{p+1}$ is valid, which implies that $\Phi_{p+1}[h]<0$. By the increasing-decreasing law, $\Phi_k[h]<0$. And now  the third inequality follows  from~\eqref{Phi_k&Delta}.

\end{itemize}

    \item $\Psi'_{k+1}[h]\leq 0$.

   It follows from~\eqref{Phi_k&Delta} since $\Phi_{k}[h]<0$, as we have just seen above.

\end{itemize}

\end{enumerate}

So we have seen that all the moves of sequence~\eqref{invalidWithoutGap2} are valid  with the exception of $T_2^g$. On the other hand, as we pointed out in the paragraph just after sequence~\eqref{invalidWithoutGap}, we have that $j\neq g$,  the $j^{th}$ and the $g^{th}$ position in $\Phi_p$ belong to the same tuple, and $\Phi_p[j]=\Phi_p[g]$. This means that we can substitute  the pair of moves $$\Phi_p \xrightarrow{T^{j}_2}\Psi_{p+1}\xrightarrow{\text{\circled{$T_2^{g}$}}}  \Psi'_{p+2}$$ in  \eqref{invalidWithoutGap2} by the valid move  $T^{jg}_6$   to obtain the sequence of valid moves:
\begin{equation*} 
   \Phi_p \xrightarrow{T^{jg}_6} \Psi'_{p+2} \xrightarrow{M_{p+1}}\Psi'_{p+3} \xrightarrow{M_{p+2}}  \cdots\xrightarrow{M_{k-1}} \Psi'_{k+1}  \xrightarrow{T^{ih}_3} \Lambda. 
\end{equation*}

Then we need to repeat our analysis of the original sequence \eqref{T6} with the shorter sequence
\begin{equation*} 
   \Phi_1 \xrightarrow{M_1}  \cdots\xrightarrow{M_{p-1}} \Phi_p \xrightarrow{T^{jg}_6} \Psi'_{p+2}.
\end{equation*}
Since $\Phi_p[j]=\Phi_p[g]$ then we can proceed as in the subcase 6A, 
 where we proved that $\Psi'_{p+2}$ is $C^*_{\mathbb{Z},3}-$realizable. And so $\Lambda$ is also $C^*_{\mathbb{Z},3}-$realizable.

\end{description}
\end{description}
\end{description}
\end{proof}


\section{Final remarks} \label{Final_remarks}

The main theorem of this work and the techniques developed, seem to open an array of interesting questions. The most obvious question is to decide if the main theorem can be extended to rationals and to reals. On the other hand,  the simplification of the $C-$realizability to just three moves, and the  technique of swapping moves that has been extensively used, 
poses the question of finding a  way of reorganizing the moves to obtain an order that can be considered canonical. In subsection~\ref{reordering-moves} we propose such a canonical order.  

Also, associated to the $C-$realizability we have a decision problem. That is, given a multiset, is it $C-$realizable? In~\cite{bc2017}, we studied this problem and concluded that it is NP-hard. But in this work we have considered states, or ordered lists. So, we might ask about the complexity of the decision problem of knowing if a given state is $C^*_{\mathbb{Z},3}-$realizable or not. After considering the canonical order that was mentioned in the previous paragraph, it might happen that the decision problem associated to states has polynomial complexity, and that the NP-hardness of the decision problem for multisets comes from factorial amount of ways to order a multiset.

To conclude, we will talk more about some of the questions raised.


\subsection{Extension of the main theorem to rationals and to reals} \label{extension-rationals-reals}

We would like to extend the main result of this work, Theorem~\ref{avoidingtype456}, to rationals. That is, to prove the equivalence between  $C_{\mathbb{Q},6}$ and $C_{\mathbb{Q},3}-$realizability. 
First, note that  $C_{\mathbb{Z},3}-$realizability implies  $C_{\mathbb{Q},3}-$realizability  in the following sense.

\begin{lemma} \label{CZ_realizability_to_CQ}
Let $\{\lambda_1,\ldots,\lambda_n\}$ be a multiset of integers numbers that is $C_{\mathbb{Z},3}-$realizable, then for any integer $q$ the multiset  $\{\frac{\lambda_1}{q},\ldots,\frac{\lambda_n}{q}\}$  is $C_{\mathbb{Q},3}-$realizable.
\end{lemma}
\begin{proof}
   Is  trivial, since all the process that start with $\{0\},\ldots,\{0\}$ and finish with $\{\lambda_1,\ldots,\lambda_n\}$ by applying a sequence of rules of type~\eqref{move-i}--\eqref{move-iii} of Theorem~\ref{guo-etc-2} can be reproduced introducing a factor $\frac1q$ in all the rules of type~\eqref{move-ii} and \eqref{move-iii}. 
\end{proof}

 On the other hand, if a multiset  $\Lambda$ is  $C_{\mathbb{Q},3}-$realizable by a a sequence of rules of type~\eqref{move-i-guo}--\eqref{move-iii-guo} of Theorem~\ref{guo-etc}   then there exists  an integer $q$ that depends on the entries of $\Lambda$ and on the rules applied such that $q \cdot \Lambda$ is $C_{\mathbb{Z},3}-$realizable.
 
\begin{lemma} \label{CQ_realizability_Q}
Let $\Lambda=\{\frac{\lambda_1}{\alpha_1},\ldots,\frac{\lambda_n}{\alpha_n}\}$ be a multiset of rational numbers that is $C_{\mathbb{Q},3}-$realizable by the sequence $M_1,\ldots,M_r$ of  rules of type~\eqref{move-i-guo}--\eqref{move-iii-guo} of Theorem \ref{guo-etc}. Let $M_{i_1}, \ldots, M_{i_s}$ be the subsequence of all the rules of types \eqref{move-ii-guo} or \eqref{move-iii-guo}, and let $\frac{\mu_1}{\beta_1},\ldots,\frac{\mu_s}{\beta_s}$ be the $\epsilon$'s involved in these rules. 
If $q=lcm(\alpha_1,\ldots,\alpha_n,\beta_1,\ldots,\beta_s)$ then $\Lambda'=\{q\cdot \frac{\lambda_1}{\alpha_1},\ldots,q\cdot\frac{\lambda_n}{\alpha_n}\}$ is $C_{\mathbb{Z},3}-$realizable.
\end{lemma} 

\begin{proof}
We will construct a sequence of rules $M'_1,\ldots, M'_r$ of type~\eqref{move-i-guo}--\eqref{move-iii-guo} of Theorem \ref{guo-etc} for integers based on the given sequence.  If $M_j$ is a rule of type~\eqref{move-i-guo}, let $M'_j=M_j$. For $k=1,\ldots,s$ substitute  the rule $M_{i_k}$  by  rule $M'_{i_k}$ where $\epsilon_{i_k}$ becomes $q\cdot \epsilon_{i_k}$. Then $\Lambda'$ is   $C_{\mathbb{Z},3}-$realizable by the sequence of rules $M'_1,\ldots,M'_r$.
\end{proof}

In the premises of Lemma~\ref{CQ_realizability_Q} we know that  $\{\frac{\lambda_1}{\alpha_1},\ldots,\frac{\lambda_n}{\alpha_n}\}$ is $C_{\mathbb{Q},3}-$realizable  by means of the concrete rules $M_1,\ldots,M_r$ and so we know the $\epsilon$'s involved in the rules. These   $\epsilon$'s  have an important role in the conclusion, and this is what we want to avoid.
That is, the only premise we want to start from is that $\{\frac{\lambda_1}{\alpha_1},\ldots,\frac{\lambda_n}{\alpha_n}\}$ is $C_{\mathbb{Q},3}-$realizable, and the  optimal result that we wish to achieve  is:

\begin{conjecture} \label{Conj_CQ_realizability}
Let $\{\frac{\lambda_1}{\alpha_1},\ldots,\frac{\lambda_n}{\alpha_n}\}$ be a multiset  $C_{\mathbb{Q},3}-$realizable and let $p=lcm(\alpha_1,\ldots,\alpha_n)$. Then $\{\frac{\lambda_1}{\alpha_1},\ldots,\frac{\lambda_n}{\alpha_n}\}$  is $C_{\mathbb{Q},3}-$realizable if and only if $\{p\cdot \frac{\lambda_1}{\alpha_1},\ldots,p\cdot\frac{\lambda_n}{\alpha_n}\}$ is $C_{\mathbb{Z},3}-$realizable.
\end{conjecture} 

Lemma~\ref{CQ_realizability_Q} says that $\{q\cdot \frac{\lambda_1}{\alpha_1},\ldots,q\cdot\frac{\lambda_n}{\alpha_n}\}$ is $C_{\mathbb{Z},3}-$realizable being $q$ a multiple of $lcm(\alpha_1,\ldots,\alpha_n)$. 
So to prove the conjecture, we would only need to prove that for any integer $m>0$: 
\[\text{if $\{m \lambda_1,\ldots,m \lambda_n\}$  is  $C_{\mathbb{Z},3}-$realizable then $\{\lambda_1,\ldots, \lambda_n\}$ is $C_{\mathbb{Z},3}-$realizable.}\]
We have not been able to prove this statement.

We also haven't been able to extend the main theorem to reals. We have tried to convert the reasoning in the proof of Theorem~\ref{avoidingtype456} to the rules of type~\eqref{move-i-guo}--\eqref{move-vi-guo} of Theorem~\ref{guo-etc} with general $\epsilon$, but we encountered technical difficulties.

\subsection{Reordering the moves} \label{reordering-moves}

As it was explained in Lemma~\ref{well_defined_swap} and the example after the Lemma, reordering moves of type 1 is troublesome. So these moves should remain fixed as much as possible.

Given a sequence of valid moves that $C^*_{\mathbb{Z},3}-$realize a certain state,  Corollary~\ref{T2_begining} tells us that we can take all the moves of type 2 to the beginning so that the moves of the whole sequence remain valid, and the final state will still be the same.

Once we put all the type 2 moves at the beginning of the sequence we obtain
\begin{align*} \label{orig-sequence-2Lambda-withT2-beginning}
    ((0),\ldots,(0)) \underbrace{\xrightarrow{T_2^{1}} \cdots \xrightarrow{T_2^{1}}}_{a_1} \underbrace{\xrightarrow{T_2^{2}} \cdots \xrightarrow{T_2^{2}}}_{a_2} \cdots  \underbrace{\xrightarrow{T_2^{n}} \cdots \xrightarrow{T_2^{n}}}_{a_n}\ \Big((a_1), (a_2), \ldots, (a_n)\Big) \longrightarrow  \cdots  \longrightarrow  \Lambda.
\end{align*}
Now we can pull towards the beginning each move of type 3 as much as possible. Note  that we can swap any two consecutive valid moves of type 3, just arguing as we did in Lemma~\ref{swapT2} with valid moves of type 2. When we have a valid move of type 1 followed by a valid move $T_3^{ij}$ of type 3 we can swap them so that both remain valid with one exception: if the type 1 move joins two tuples, one containing the $i^{th}-$position and the other one the $j^{th}-$position.


The final structure of the sequence  after all these swaps is a new sequence of valid moves where first appears the moves of type 2. Then a move $M_1$ of type 1 that joins for instance tuples $A_1$ and $B_1$, then any moves $T_3^{ij}$ for which the $i^{th}-$position belongs $A_1$ and the $j^{th}-$position belongs to $B_1$, or vice-versa. Then another move $M_2$ of type 1 that joins for instance tuples $A_2$ and $B_2$, then any move $T_3^{ij}$ for which the $i^{th}-$position belongs to  $A_2$ and the $j^{th}-$position belongs to $B_2$, or vice-versa. And so on. The final state would remain being $\Lambda$.

We believe that this order of moves for $C^*_{\mathbb{Z},3}-$realizable states can simplify the search of an algorithm that solves the decision problem of knowing if a given state is $C^*_{\mathbb{Z},3}-$realizable or not.


\bigskip

{\bf Acknowledgements:} This research has been funded by the Agencia Estatal de Investigaci\'on of Spain through
Grant PID2019-106362GB-I00/AEI/10.13039/501100011033.


\begin{thebibliography}{99}


\bibitem{bc2017} A.\ Borobia, R.\ Canogar.
\newblock The real nonnegative inverse eigenvalue problem is NP-hard.
\newblock Linear Algebra Appl., 522 (2017) 127--139.

\newblock https://doi.org/10.1016/j.laa.2017.02.010


\bibitem{bms2004} A. Borobia, J. Moro, R.L. Soto.
\newblock Negativity compensation in the nonnegative inverse eigenvalue problem.
\newblock Linear Algebra Appl. 393 (2004) 73--89.

\newblock https://doi.org/10.1016/j.laa.2003.10.023

\bibitem{bms2008} A. Borobia, J. Moro, R.L. Soto.
\newblock A unified view on compensation criteria in the real nonnegative inverse eigenvalue problem.
\newblock Linear Algebra Appl. 428 (2008) 2574--2584.

\newblock https://doi.org/10.1016/j.laa.2007.11.031

\bibitem{es2015} R. Ellard, H. \v{S}migoc.
\newblock Connecting sufficient conditions for the Symmetric Nonnegative Inverse Eigenvalue Problem.
\newblock Linear Algebra Appl. 498 (2016) 521--552.

\newblock http://dx.doi.org/10.1016/j.laa.2015.10.035

\bibitem{enn1998} L. Elsner, R. Nabben, M. Neumann.
\newblock Orthogonal Bases that Lead to Symmetric Nonnegative Matrices.
\newblock Linear Algebra Appl. (1998) 271 (1998) 323--343.

\newblock https://doi.org/10.1016/S0024-3795(97)00302-9

\bibitem{guo1997} W. Guo.
\newblock Eigenvalues of nonnegative matrices. 
\newblock Linear Algebra Appl., 266 (1997) 261--270.

\newblock https://doi.org/10.1016/S0024-3795(96)00007-9

\bibitem{jmpp} 
C.R. \ Johnson, C. \ Mariju\'an, P. \ Paparella, M. \ Pisonero. 
\newblock The NIEP.
\newblock In: Andr\'e, C., Bastos, M., Karlovich, A., Silbermann, B., Zaballa, I. (eds).
\newblock Operator Theory, Operator Algebras, and Matrix Theory. Operator Theory: Advances and Applications, vol 267 (2018) 199--220. 
\newblock Birkh\"{a}user, Cham. 

\newblock https://doi.org/10.1007/978-3-319-72449-2\_10


\bibitem{mm2021} C.\ Mariju\'an, J.\ Moro.
\newblock A characterization of trace-zero sets realizable by compensation in the SNIEP. 
\newblock Linear Algebra Appl., 615 (2021) 42--76. 

\newblock https://doi.org/10.1016/j.laa.2020.12.021



\bibitem{mm2023} C.\ Mariju\'an, J.\ Moro.
\newblock A characterization of sets realizable by compensation in the SNIEP. \newblock Preprint 2023.



\bibitem{mps1} C. Mariju\'an, M. Pisonero, R.L. Soto.
\newblock A map of sufficient conditions for the symmetric nonnegative inverse eigenvalue problem.
\newblock Linear Algebra and its Applications, Volume 530 (2017) 344--365.

\newblock https://doi.org/10.1016/j.laa.2017.05.023

\bibitem{mps2} C. \ Mariju\'an, M. \ Pisonero, R.L. \ Soto. \newblock Updating a map of sufficient conditions for the real nonnegative inverse eigenvalue problem.
\newblock Special Matrices, vol. 7, no. 1 (2019) pp. 246-256. 

\newblock https://doi.org/10.1515/spma-2019-0018


\bibitem{s2013} R.L. Soto.
\newblock  A family of realizability criteria for the real and symmetric nonnegative inverse eigenvalue
problem.
\newblock  Numer. Linear Algebra Appl. 20 (2) (2013) 336--348.

\newblock https://doi.org/10.1002/nla.835

\bibitem{s1983} G.W. Soules. 
\newblock Constructing symmetric nonnegative matrices. \newblock Linear and Multilinear Algebra. 13 (1983) 241--251.

\newblock https://doi.org/10.1080/03081088308817523

\end{thebibliography}
\end{document}